\documentclass[11pt,fleqn]{article}

\usepackage{amsfonts}
\usepackage{amsmath}
\usepackage{amssymb}
\usepackage{amsthm}
\usepackage{bbm}
\usepackage{bm}
\usepackage{booktabs}
\usepackage{cases}
\usepackage{cite}
\usepackage{graphicx}
\usepackage{indentfirst}
\usepackage{longtable}
\usepackage{makecell}
\usepackage{mathrsfs}
\usepackage{titlesec}
\usepackage{rotating}
\usepackage[T1]{fontenc}
\usepackage[centerlast]{caption}

\allowdisplaybreaks

\numberwithin{equation}{section}

\newtheorem{theorem}{Theorem}[section]
\newtheorem{lemma}[theorem]{Lemma}

\newtheorem{corollary}[theorem]{Corollary}

\theoremstyle{definition}
\newtheorem{definition}{Definition}[section]
\newtheorem*{remark}{Remark}
\newtheorem{example}{Example}[section]

\titleformat{\section}
{\normalfont\normalsize\bf}{\thesection.}{6pt}{}
\titleformat{\subsection}
{\normalfont\normalsize}{\thesubsection.}{6pt}{}
\titleformat{\subsubsection}
{\normalfont\normalsize}{\thesubsubsection.}{6pt}{}

\newcommand{\citu}[2]{\!\!\cite[#2]{#1}}

\newcommand{\bibb}[1]{\left\{#1\right\}}

\newcommand{\smbb}[1]{\left(#1\right)}

\newcommand{\tf}{\tfrac}

\newcommand{\al}{\alpha}
\newcommand{\be}{\beta}
\newcommand{\ga}{\gamma}
\newcommand{\de}{\delta}

\newcommand{\la}{\lambda}
\newcommand{\si}{\sigma}

\newcommand{\ze}{\zeta}
\newcommand{\om}{\omega}

\newcommand{\vPsi}{\varPsi}

\newcommand{\ud}{\mathrm{d}}

\newcommand{\ui}{\mathrm{i}}
\newcommand{\Li}{\mathrm{Li}}
\newcommand{\Lii}{\,\mathrm{Li}}

\newcommand{\imm}{\,\mathrm{Im}}

\newcommand{\ol}{\overline}

\def\a{^{(A)}}
\def\b{^{(B)}}
\def\c{^{(C)}}

\def\t{\tilde{t}}
\def\S{\tilde{S}}

\makeatletter
\newcommand{\tmod}[1]{{\@displayfalse\pmod{#1}}}
\makeatother

\makeatletter

\newdimen\bibspace
\renewenvironment{thebibliography}[1]{%
 \section*{\refname 
       \@mkboth{\MakeUppercase\refname}{\MakeUppercase\refname}}%
     \list{\@biblabel{\@arabic\c@enumiv}}%
          {\settowidth\labelwidth{\@biblabel{#1}}%
           \leftmargin\labelwidth
           \advance\leftmargin\labelsep
           \itemsep\bibspace
           \parsep\z@skip     %
           \@openbib@code
           \usecounter{enumiv}%
           \let\p@enumiv\@empty
           \renewcommand\theenumiv{\@arabic\c@enumiv}}%
     \sloppy\clubpenalty4000\widowpenalty4000%
     \sfcode`\.\@m}
    {\def\@noitemerr
      {\@latex@warning{Empty `thebibliography' environment}}%
     \endlist}

\makeatother

\begin{document}

\title{\bf\boldmath{Dirichlet type extensions of Euler sums}}
\author
{
{
Ce Xu$\,^{a,}$\thanks{\emph{E-mail address\,:} cexu2020@ahnu.edu.cn (Ce Xu).}
\quad
Weiping Wang$\,^{b,}$\thanks{Corresponding author. \emph{E-mail addresses\,:}
wpingwang@126.com, wpingwang@zstu.edu.cn (Weiping Wang).}
}\\
\small $a.$ School of Mathematics and Statistics, Anhui Normal University,
    Wuhu 241002, P.R. China\\
\small $b.$ School of Science, Zhejiang Sci-Tech University,
    Hangzhou 310018, P.R. China
}

\date{}
\maketitle

\vspace{-0.5cm}
\begin{center}
\parbox{6.3in}{\small{\bf Abstract}\vspace{3pt}

\hspace{3.5ex}In this paper, we study the alternating Euler $T$-sums and $\S$-sums, which are infinite series involving (alternating) odd harmonic numbers, and have similar forms and close relations to the Dirichlet beta functions. By using the method of residue computations, we establish the explicit formulas for the (alternating) linear and quadratic Euler $T$-sums and $\S$-sums, from which, the parity theorems of Hoffman's double and triple $t$-values and Kaneko-Tsumura's double and triple $T$-values are further obtained. As supplements, we also show that the linear $T$-sums and $\S$-sums are expressible in terms of colored multiple zeta values. Some interesting consequences and illustrative examples are presented.
}

\vspace{6pt}
\parbox{6.3in}{\small{\emph{AMS classification\,:}}\,
11A07; 11M32; 40A25}

\vspace{3pt}
\parbox{6.3in}{\small{\emph{Keywords\,:}}\,
Multiple zeta values; Euler sums; Hoffman $t$-values; Kaneko-Tsumura $T$-values; Contour integration; Residue theorem}
\end{center}


\setcounter{tocdepth}{2}
\tableofcontents


\section{Introduction}\label{Sec.Intro}

The classical Euler sums are of the form
\[
S_{p_1p_2\cdots p_k,q}
    =\sum_{n=1}^\infty\frac{H_n^{(p_1)}H_n^{(p_2)}\cdots H_n^{(p_k)}}{n^q}\,,
\]
where $p_1,p_2,\ldots,p_k,q$ are positive integers with $p_1\leq p_2\leq\cdots\leq p_k$ and $q>1$. Here $H_n^{(r)}$ stand for the \emph{generalized harmonic numbers}, which are defined by $H_0^{(r)}=0$ and $H_n^{(r)}=\sum_{k=1}^n1/k^r$, for $n,r=1,2,\ldots$, with $H_n\equiv H_n^{(1)}$. This kind of infinite series can be traced back to Euler and Goldbach (see Berndt \cite[p. 253]{Berndt85.1}). They are essential to the connection of knot theory with quantum field theory \cite{Br2013,Kassel95}, and became even more important when higher order calculations in quantum electrodynamics (QED) and quantum chromodynamics (QCD) started to need the multiple harmonic sums \cite{Bl1999,BBV2010}. For an brief introduction of more applications, the readers may consult in Flajolet and Salvy \cite[Section 1]{FlSa98}.

In the previous paper \cite{XuWang19.TVE}, we introduced and studied the following two variants of the classical Euler sums:
\begin{equation}\label{Euler.TS}
T_{p_1p_2\cdots p_k,q}
    =\sum_{n=1}^\infty\frac{h_{n-1}^{(p_1)}h_{n-1}^{(p_2)}\cdots h_{n-1}^{(p_k)}}{(n-1/2)^q}
\quad\text{and}\quad
\S_{p_1p_2\cdots p_k,q}
    =\sum_{n=1}^\infty\frac{h_n^{(p_1)}h_n^{(p_2)}\cdots h_n^{(p_k)}}{n^q}\,,
\end{equation}
where $h_n^{(r)}$ are the \emph{odd harmonic numbers of order $r$} defined by
\[
h_0^{(r)}=0\quad\text{and}\quad
h_n^{(r)}=\sum_{k=1}^n\frac{1}{(k-1/2)^r}\,,\quad\text{for } n,r=1,2,\ldots,
\]
with $h_n\equiv h_n^{(1)}$. The quantity $w=p_1+\cdots+p_k+q$ is called the \emph{weight} of the corresponding sum, and the quantity $k$ is called the \emph{degree} (or \emph{order}). We call these two kinds of sums the \emph{Euler $T$-sums} and the \emph{Euler $\S$-sums}, respectively. Similarly to the classical case, the Euler $T$-sums and $\S$-sums can be evaluated by the method of contour integral representations and residue computations developed by Flajolet and Salvy \cite{FlSa98}. As a consequence, we established in \cite{XuWang19.TVE} many explicit evaluations of the Euler $T$-sums and $\S$-sums via $\ln(2)$, multiple zeta values and Hoffman's multiple $t$-values. Besides the $T$-sums and $\S$-sums, some other general Euler type sums were discussed systematically in \cite{X2020.EFG,X2020.EFS}.

Here, for $(s_1,s_2,\ldots,s_k)\in\mathbb{N}^k$ with $s_1>1$, the \emph{multiple zeta values} (MZVs) \cite{Hoff92,Zag92} and the \emph{multiple $t$-values} (MtVs) \cite{Hoff19.AOV} are defined by
\[
\ze(s_1,s_2,\ldots,s_k)=\sum_{n_1>n_2>\cdots>n_k\geq 1}
    \frac{1}{n_1^{s_1}n_2^{s_2}\cdots n_k^{s_k}}
\]
and
\begin{align*}
t(s_1,s_2,\ldots,s_k)
&=\sum_{\substack{n_1>n_2>\cdots>n_k\geq 1\\n_i\text{ odd}}}
    \frac{1}{n_1^{s_1}n_2^{s_2}\cdots n_k^{s_k}}\\
&=\sum_{n_1>n_2>\cdots>n_k\geq 1}\frac{1}{(2n_1-1)^{s_1}(2n_2-1)^{s_2}\cdots(2n_k-1)^{s_k}}\,,
\end{align*}
respectively. As a normalized version of the multiple $t$-values, we call
\begin{align*}
\t(s_1,s_2,\ldots,s_k)
&=\sum_{n_1>n_2>\cdots>n_k\geq 1}\frac{1}{(n_1-1/2)^{s_1}(n_2-1/2)^{s_2}\cdots(n_k-1/2)^{s_k}}\\
&=2^{s_1+s_2+\cdots+s_k}t(s_1,s_2,\ldots,s_k)
\end{align*}
the \emph{multiple $\t$-values}. According to the definitions, $\t(s)=2^st(s)=\ze(s,\frac{1}{2})=(2^s-1)\ze(s)$ for integer $s\geq 2$, where $\ze(s,a)$ is the \emph{Hurwitz zeta function} and $\ze(s)$ is the \emph{Riemann zeta function}.

Moreover, in the above definitions, we put a bar on the top of $s_j$ for $j=1,2,\ldots k$ if there is a sign $(-1)^{n_j}$ appearing in the numerator on the right. Those involving one or more of the $s_j$ barred are called the \emph{alternating multiple values}. For example,
\[
t(\bar{s}_1,\bar{s}_2,s_3,s_4)=\sum_{n_1>n_2>n_3>n_4\geq1}
    \frac{(-1)^{n_1+n_2}}{(2n_1-1)^{s_1}(2n_2-1)^{s_2}(2n_3-1)^{s_3}(2n_4-1)^{s_4}}\,.
\]
In particular, we have
\begin{equation}\label{AZV}
\ze(\bar{s})=\sum_{n=1}^{\infty}\frac{(-1)^n}{n^s}=(2^{1-s}-1)\ze(s)\quad\text{and}\quad
t(\bar{s})=\sum_{n=1}^\infty\frac{(-1)^n}{(2n-1)^s}=-\be(s)\,,
\end{equation}
with $\ze(\bar{1})=-\ln(2)$. Here $\be(s)$ is the \emph{Dirichlet beta function} satisfying
\[
\be(2)=G\,,\quad
\be(2k+1)=\frac{(-1)^kE_{2k}}{2^{2k+2}(2k)!}\pi^{2k+1}\,,\quad\text{for }k=0,1,2,\ldots,
\]
where $G$ represents \emph{Catalan's constant}, and $E_n$ are the \emph{Euler numbers}. For simplicity, in the present paper, for $(s_1,s_2,\ldots,s_k)\in\mathbb{N}^k$ and $(\si_1,\si_2,\ldots,\si_k)\in\{\pm1\}^k$ with $(s_1,\si_1)\neq (1,1)$, we denote the (alternating) multiple zeta values and multiple $t$-values in the following unified forms:
\[
\ze(s_1,s_2,\ldots,s_k;\si_1,\si_2,\ldots,\si_k)
    =\sum_{n_1>n_2>\cdots>n_k\geq 1}\frac{\si_1^{n_1}\si_2^{n_2}\cdots\si_k^{n_k}}
        {n_1^{s_1}n_2^{s_2}\cdots n_k^{s_k}}
\]
and
\begin{align*}
\t(s_1,s_2,\ldots,s_k;\si_1,\si_2,\ldots,\si_k)
    &=\sum_{n_1>n_2>\cdots>n_k\geq 1}\frac{\si_1^{n_1}\si_2^{n_2}\cdots\si_k^{n_k}}
        {(n_1-1/2)^{s_1}(n_2-1/2)^{s_2}\cdots(n_k-1/2)^{s_k}}\\
    &=2^{s_1+s_2+\cdots+s_k}t(s_1,s_2,\ldots,s_k;\si_1,\si_2,\ldots,\si_k)\,,
\end{align*}
respectively.

Because of the congruence condition in the summation of the MtV $t(s_1,s_2,\ldots,s_k)$ and of the fact that this value can be expressed as a linear combination of alternating MZVs, Kaneko and Tsumura \cite{KaTs19} regarded this value as a \emph{MZV of level two}. In fact, they \cite{KaTs18,KaTs19} introduced and studied another variant of MZVs of level two:
\begin{align*}
T(s_1,s_2,\ldots,s_k)
&=2^k\sum_{\substack{n_1>n_2>\cdots>n_k\geq1\\n_i\equiv k-i+1\tmod{2}}}
    \frac{1}{n_1^{s_1}n_2^{s_2}\cdots n_k^{s_k}}\\
&=2^k\sum_{n_1>n_2>\cdots>n_k\geq 1}
    \frac{1}{(2n_1-k)^{s_1}(2n_2-k+1)^{s_2}\cdots(2n_k-1)^{s_k}}\,,
\end{align*}
which are called the \emph{multiple $T$-values} (MTVs). Similarly, the corresponding alternating MTVs are defined by
\begin{align*}
&T(s_1,s_2,\ldots,s_k;\si_1,\si_2,\ldots,\si_k)\\
&\quad=2^k\sum_{n_1>n_2>\cdots>n_k\geq 1}
    \frac{\si_1^{n_1}\si_2^{n_2}\cdots\si_k^{n_k}}
        {(2n_1-k)^{s_1}(2n_2-k+1)^{s_2}\cdots(2n_k-1)^{s_k}}\,.
\end{align*}

In this paper, we study mainly the alternating Euler $T$-sums and alternating Euler $\S$-sums, which are infinite series involving (alternating) odd harmonic numbers, and can be seen as the Dirichlet type extensions of the classical Euler sums, for the similar forms and close relations to the Dirichlet beta function.

The \emph{alternating harmonic numbers} ${\bar H}^{(p)}_n$ and the \emph{alternating odd harmonic numbers} ${\bar h}^{(p)}_n$ are defined by
\[
\bar{H}_0^{(r)}=0\,,\quad\bar{H}_n^{(r)}=\sum_{k=1}^n\frac{(-1)^{k-1}}{k^r}\,,
\quad\text{and}\quad
\bar{h}_0^{(r)}=0\,,\quad\bar{h}_n^{(r)}=\sum_{k=1}^n\frac{(-1)^{k-1}}{(k-1/2)^r}\,,
\]
respectively, for $n,r=1,2,\ldots$. In the definitions of the Euler $T$-sums and $\S$-sums, if we replace $h_n^{(p_j)}$ by $\bar{h}_n^{(p_j)}$ in the numerator of the summand, we put a bar on the top of parameter $p_j$. Additionally, if there is a sign $(-1)^{n-1}$ appearing in the summand, we put a bar on the top of parameter $q$. For example, following such conventions, we have
\[
T_{p_1\bar{p}_2\bar{p}_3,\bar{q}}
    =\sum_{n=1}^\infty(-1)^{n-1}
        \frac{h_{n-1}^{(p_1)}\bar{h}_{n-1}^{(p_2)}\bar{h}_{n-1}^{(p_3)}}{(n-1/2)^q}
        \quad\text{and}\quad
\S_{p_1\bar{p}_2\bar{p}_3\bar{p}_4,q}
    =\sum_{n=1}^\infty\frac{h_n^{(p_1)}{\bar h}_n^{(p_2)}\bar{h}_n^{(p_3)}\bar{h}_n^{(p_4)}}{n^q}\,.
\]
These two sums and some similar ones with one or more of the $p_j$ or $q$ barred are called the \emph{alternating Euler $T$-sums} and \emph{alternating Euler $\S$-sums}, respectively. For convenience, for $(p_1,p_2,\ldots,p_k,q)\in\mathbb{N}^{k+1}$ and $(\si_1,\si_2,\ldots,\si_k,\si)\in\{\pm 1\}^{k+1}$ with $(q,\si)\neq (1,1)$, denote the alternating case and non-alternating case of the above variants of Euler sums in unified forms:
\begin{align}
&T_{p_1,p_2,\ldots,p_k,q}^{\si_1,\si_2,\ldots,\si_k,\si}
    =\sum_{n=1}^\infty\si^{n-1}\frac{h_{n-1}^{(p_1)}(\si_1)h_{n-1}^{(p_2)}(\si_2)\cdots
        h_{n-1}^{(p_k)}(\si_k)}{(n-1/2)^q}\,,\label{Tsum.Unify}\\
&\S_{p_1,p_2,\ldots,p_k,q}^{\si_1,\si_2,\ldots,\si_k,\si}
    =\sum_{n=1}^\infty\si^{n-1}\frac{h_n^{(p_1)}(\si_1)h_n^{(p_2)}(\si_2)\cdots
        h_n^{(p_k)}(\si_k)}{n^q}\,,\label{Stsum.Unify}
\end{align}
where the numbers $h_n^{(p)}(\si)$ denote the odd harmonic numbers $h_n^{(p)}$ for $\si=1$ and the alternating odd harmonic numbers $\bar{h}_n^{(p)}$ for $\si=-1$. The paper is organized as follows.

In Section \ref{Sec.Notation.Exp}, we define the notations of some sums, and use these sums to give series expansions of parametric digamma function and parametric cotangent function. These sums and series expansions are further used in Sections \ref{Sec.lin}--\ref{Sec.Remark} to study the alternating Euler $T$-sums and alternating Euler $\S$-sums.

In particular, in Sections \ref{Sec.lin} and \ref{Sec.qua}, by applying the method of residue computations to specific kernel functions and base functions, we establish four general identities on sums defined in Section \ref{Sec.Notation.Exp}, which reduce to expressions of the linear and quadratic (alternating) Euler $T$-sums and $\S$-sums. Further special cases are discussed, and the evaluations of many specific Euler $T$-sums and $\S$-sums are presented. By establishing the transformation formulas between various Euler sums and multiple values, the parity theorems of Hoffman's double and triple $t$-values, and Kaneko-Tsumura's double and triple $T$-values are also established.

Additionally, as supplements, we introduce in Section \ref{Sec.lin} some necessary results on the linear Euler $R$-sums, which are used in the discussion of the quadratic Euler $\S$-sums as well as the general $\S$-sums. We show that the linear Euler $T,\S,R$-sums are expressible in terms of the colored multiple zeta values of level four, and determine the evaluations of more specific Euler sums, which can not be obtained by the method of residue theorem.

Finally, in Section \ref{Sec.Remark}, we give some remarks, and present the parity theorems of the general (alternating) Euler $T$-sums and $\S$-sums.


\section{Notations of sums and expansions of functions}\label{Sec.Notation.Exp}

Let $A=\{a_k\}_{k\in\mathbb{Z}}$ be a real sequence satisfying $\lim_{k\to\infty}a_k/k=0$. For convenience, denote $A_1$ and $A_2$ as the constant sequence $\{1^k\}$ and the alternating sequence $\{(-1)^k\}$, respectively.

For brevity, the following notations of sums related to the sequence $A$ are used in the sequel. Let $j\geq1$ be a positive integer. Then for $n\in\mathbb{N}_0$, define
\begin{align*}
&D\a(j)=\sum_{k=1}^\infty\frac{a_k}{k^j}\,,\quad D\a(1)=0\,,\\
&E\a_n(j)=\sum_{k=1}^n\frac{a_{n-k}}{k^j}\,,\quad E\a_0(j)=0\,,&&
    \bar{E}\a_n(j)=\sum_{k=1}^n\frac{a_{k-n-1}}{k^j}\,,\quad\bar{E}\a_0(j)=0\,,\\
&\hat{E}\a_n(j)=\sum_{k=1}^n\frac{a_{n-k}}{(k-1/2)^j}\,,\quad\hat{E}\a_0(j)=0\,,&&
    \tilde{E}\a_n(j)=\sum_{k=1}^n\frac{a_{k-n-1}}{(k-1/2)^j}\,,\quad\tilde{E}\a_0(j)=0\,,
\end{align*}
and for $n\in\mathbb{Z}$, define
\begin{align*}
&F\a_n(j)=\left\{
    \begin{array}{ll}
        \sum\limits_{k=1}^\infty\frac{a_{k+n}-a_k}{k}\,,&j=1\,,\\
        \sum\limits_{k=1}^\infty\frac{a_{k+n}}{k^j}\,,&j>1\,,
    \end{array}\right.&&
\bar{F}\a_n(j)=\left\{
    \begin{array}{ll}
        \sum\limits_{k=1}^\infty\frac{a_{k-n}-a_k}{k}\,,&j=1\,,\\
        \sum\limits_{k=1}^\infty\frac{a_{k-n}}{k^j}\,,&j>1\,,\\
    \end{array}\right.\\
&\hat{F}\a_n(j)=\left\{
    \begin{array}{ll}
        \sum\limits_{k=1}^\infty\smbb{\frac{a_{k+n}}{k-1/2}-\frac{a_k}{k}}\,,&j=1\,,\\
        \sum\limits_{k=1}^\infty\frac{a_{k+n}}{(k-1/2)^j}\,,&j>1\,,
    \end{array}\right.&&
\tilde{F}\a_n(j)=\left\{
    \begin{array}{ll}
        \sum\limits_{k=1}^\infty\smbb{\frac{a_{k-n}}{k-1/2}-\frac{a_k}{k}}\,,&j=1\,,\\
        \sum\limits_{k=1}^\infty\frac{a_{k-n}}{(k-1/2)^j}\,,&j>1\,.
    \end{array}\right.
\end{align*}
Now, let
\begin{align*}
&G\a_n(j)=E\a_n(j)-\bar{E}\a_{n-1}(j)-\tf{a_0}{n^j}\,,&& G\a_0(j)=0\,,\\
&L\a_n(j)=F\a_n(j)+(-1)^j\bar{F}\a_n(j)\,,&&\\
&M\a_n(j)=E\a_n(j)+(-1)^jF\a_n(j)\,,&&
    \bar{M}\a_n(j)=\bar{F}\a_n(j)-\bar{E}\a_{n-1}(j)\,,\\
&N\a_n(j)=\hat{E}\a_n(j)+(-1)^j\hat{F}\a_{n-1}(j)\,,&&
    \bar{N}\a_n(j)=\tilde{F}\a_n(j)-\tilde{E}\a_{n-1}(j)\,,\\
&R\a_n(j)=G\a_n(j)+(-1)^j L\a_n(j)\,,&&
    S\a_n(j)=N\a_n(j)+\bar{N}\a_n(j)-\tf{a_0}{(n-1/2)^j}\,,
\end{align*}
and let
\begin{align*}
&\hat{t}\a(j)=\left\{
    \begin{array}{ll}
        \sum\limits_{k=1}^\infty\smbb{\frac{a_{k-1}}{k-1/2}-\frac{a_k}{k}}\,,&j=1\,,\\
        \sum\limits_{k=1}^\infty\frac{a_{k-1}}{(k-1/2)^j}\,,&j>1\,,
    \end{array}\right.\quad
\t\a(j)=\left\{
    \begin{array}{ll}
        \sum\limits_{k=1}^\infty\smbb{\frac{a_{k}}{k-1/2}-\frac{a_k}{k}}\,,&j=1\,,\\
        \sum\limits_{k=1}^\infty\frac{a_{k}}{(k-1/2)^j}\,,&j>1\,,
    \end{array}\right.\\
&\check{t}\a(j)=(-1)^{j-1}\hat{t}\a(j)-\t\a(j)\,.
\end{align*}
Note that in the above definitions, setting $A=A_1$ or $A_2$ yields
\begin{align*}
&M^{(A_1)}_n(j)=H^{(j)}_n+(-1)^j\ze(j)\,,\quad
    \bar{M}^{(A_1)}_n(j)=\ze(j)-H^{(j)}_{n-1}\,,\\
&N^{(A_1)}_n(j)=h^{(j)}_n+(-1)^j\t(j)\,,\quad
    \bar{N}^{(A_1)}_n(j)=\t(j)-h^{(j)}_{n-1}\,,\\
&R^{(A_1)}_n(j)=(1+(-1)^j)\ze(j)\,,\quad
    S^{(A_1)}_n(j)=(1+(-1)^j)\t(j)\,,\\
&D^{(A_1)}(j)=\ze(j)\,,\quad
    \hat{t}^{(A_1)}(j)=\t^{(A_1)}(j)=\t(j)\,,\quad
    \check{t}^{(A_1)}(j)=-(1+(-1)^j)\t(j)\,,
\end{align*}
and
\begin{align*}
&M^{(A_2)}_n(j)=(-1)^{n-1}\bar{H}^{(j)}_n
    +(-1)^{n+j}\ze(\bar{j})-\de_{j,1}\ln(2)\,,\\
&\bar{M}^{(A_2)}_n(j)=(-1)^{n}\bar{H}^{(j)}_{n-1}
    +(-1)^n\ze(\bar{j})+\de_{j,1}\ln(2)\,,\\
&N^{(A_2)}_n(j)=(-1)^{n-1}\bar{h}^{(j)}_n
    +(-1)^{n+j-1}\t(\bar{j})-\de_{j,1}\ln(2)\,,\\
&\bar{N}^{(A_2)}_n(j)=(-1)^{n}\bar{h}^{(j)}_{n-1}
    +(-1)^n\t(\bar{j})+\de_{j,1}\ln(2)\,,\\
&R^{(A_2)}_n(j)=(-1)^n(1+(-1)^j)\ze(\bar{j})\,,\quad
    S^{(A_2)}_n(j)=(-1)^n(1-(-1)^j)\t(\bar{j})\,,\\
&D^{(A_2)}(j)=\ze(\bar{j})+\de_{j,1}\ln(2)\,,\quad
    \hat{t}^{(A_2)}(j)=-\t(\bar{j})+\de_{j,1}\ln(2)\,,\\
&\t^{(A_2)}(j)=\t(\bar{j})+\de_{j,1}\ln(2)\,,\quad
    \check{t}^{(A_2)}(j)=-(1-(-1)^j)\t(\bar{j})\,.
\end{align*}
where $\de_{n,k}$ is the Kronecker delta. Note also that we should interpret $\ze(1):=0$ and $\t(1):=2\ln(2)$ wherever they occur.

Some sums presented here can also be found in a previous paper \cite{X2020.EFG} of the second named author. Moreover, in \cite{X2020.EFG}, the following two definitions were introduced.

\begin{definition}\label{Def.diga}
The \emph{parametric digamma function} $\vPsi(-s;A)$ related to the sequence $A$ is defined by
\[
\vPsi(-s;A)=\frac{a_0}{s}+\sum_{k=1}^\infty\smbb{\frac{a_k}{k}-\frac{a_k}{k-s}}\,,\quad
\text{for } s\in\mathbb{C}\setminus\mathbb{N}_0\,.
\]
\end{definition}

\begin{definition}\label{Def.cot}
The \emph{parametric cotangent function} $\cot(\pi s;A)$ related to the sequence $A$ is defined by
\[
\pi\cot(\pi s;A)
    =-\frac{a_0}{s}+\vPsi(-s;A)-\vPsi(s;A)
    =\frac{a_0}{s}-2s\sum_{k=1}^\infty\frac{a_k}{k^2-s^2}\,.
\]
\end{definition}

Obviously, when $A=A_1$, the parametric digamma function $\vPsi(-s;A)$ reduces to the classical digamma function $\psi(-s)+\ga$, where $\ga$ is the Euler-Mascheroni constant. By setting $A=A_1$ or $A_2$, respectively, the parametric cotangent function $\cot(\pi s;A)$ turns into
\[
\cot(\pi s;A_1)=\cot(\pi s)\quad\text{and}\quad
\cot(\pi s;A_2)=\csc(\pi s)\,.
\]

Using these notations, we give some series expansions, which will be frequently used in the computation in the sequel. For the parametric digamma function $\vPsi(-s;A)$, we have

\begin{lemma}[\citu{X2020.EFG}{Theorem 2.2}]\label{blem1}
For positive integers $p$ and $n$, if $|s+n|<1$, then
\[
\frac{\vPsi^{(p-1)}(-s;A)}{(p-1)!}
    =(-1)^p\sum_{j=1}^\infty\binom{j+p-2}{p-1}\bar{M}\a_n(j+p-1)(s+n)^{j-1}\,.
\]
\end{lemma}

Moreover, the next two lemmas hold.

\begin{lemma}\label{Lem.Psi.nZ}
For integers $p\geq1$ and $n\geq 0$, if $|s-n|<1$, then
\begin{equation}\label{2.1}
\frac{\vPsi^{(p-1)}(\tf{1}{2}-s;A)}{(p-1)!}
    =\sum_{j=1}^\infty(-1)^{j-1}\binom{j+p-2}{p-1}N\a_n(j+p-1)(s-n)^{j-1}\,,
\end{equation}
and if $|s+n|<1$, then
\begin{equation}\label{2.2}
\frac{\vPsi^{(p-1)}(\tf{1}{2}-s;A)}{(p-1)!}
    =(-1)^p\sum_{j=1}^\infty\binom{j+p-2}{p-1}\bar{N}\a_{n+1}(j+p-1)(s+n)^{j-1}\,.
\end{equation}
\end{lemma}

\begin{proof}
This lemma follows directly from the definition of the function $\vPsi(-s;A)$.
\end{proof}

\begin{lemma}\label{Lem.Psi.hn}
For positive integers $p$ and $n$, if $0<|s-n+1/2|<1$, then
\begin{align*}
&\frac{\vPsi^{(p-1)}(\tf{1}{2}-s;A)}{(p-1)!}\\
&\quad=\frac{1}{(s-n+\tf{1}{2})^p}\smbb{
    a_{n-1}-\sum_{j=1}^\infty(-1)^j\binom{j+p-2}{p-1}M\a_{n-1}(j+p-1)(s-n+\tf{1}{2})^{j+p-1}}\,.
\end{align*}
\end{lemma}

\begin{proof}
This lemma can be obtained immediately from Definition \ref{Def.diga} or \cite[Theorem 2.1]{X2020.EFG}.
\end{proof}

Setting $n=0$ in (\ref{2.1}) and (\ref{2.2}) gives
\begin{equation}\label{Psi.s.0}
\frac{\vPsi^{(p-1)}(\tf{1}{2}-s;A)}{(p-1)!}
    =(-1)^p\sum_{j=1}^\infty\binom{j+p-2}{p-1}\hat{t}\a(j+p-1)s^{j-1}\,,
    \quad\text{for }|s|<1\,,
\end{equation}
and setting $n=1$ in Lemma \ref{Lem.Psi.hn} yields
\begin{equation}\label{Psi.s.0.5}
\frac{\vPsi^{(p-1)}(\tf{1}{2}-s;A)}{(p-1)!}
    =\frac{a_{0}}{(s-\tf{1}{2})^p}
    +(-1)^p\sum_{j=1}^\infty\binom{j+p-2}{p-1}D\a(j+p-1)(s-\tf{1}{2})^{j-1}\,.
\end{equation}

For the parametric cotangent function $\cot(\pi s;A)$, Definition \ref{Def.cot} gives
\begin{equation}\label{cot.s.0}
\pi\cot(\pi s;A)=\frac{a_0}{s}-2\sum_{j=1}^\infty D\a(2j)s^{2j-1}\,,\quad\text{for }|s|<1\,.
\end{equation}
Since $R_0\a(j)=(1+(-1)^j)D\a(j)$, it can be found that the above expansion is the $n=0$ case of the following one presented in \cite{X2020.EFG}:

\begin{lemma}[\citu{X2020.EFG}{Theorem 2.3}]\label{Lem.cot.nZ}
For integer $n$, if $0<|s-n|<1$, then
\[
\pi\cot(\pi s;A)=\frac{a_{|n|}}{s-n}-\sum_{j=1}^\infty(-\si_n)^j R\a_{|n|}(j)(s-n)^{j-1}\,,
\]
where $\si_n$ is defined by
\[
\si_n= \left\{
    \begin{array}{cl}
        1\,,&n\geq 0\,,\\
        -1\,,&n<0\,.
    \end{array}\right.
\]
\end{lemma}

Additionally, we have the next result.

\begin{lemma}\label{Lem.cot.hn}
For integers $m\geq0$ and $n\geq1$, if $|s-n+1/2|<1$, then
\[
\frac{\ud^m}{\ud s^m}(\pi\cot(\pi s;A))
    =(-1)^mm!\sum_{j=1}^\infty(-1)^{j-1}\binom{j+m-1}{m}S\a_n(j+m)(s-n+\tf{1}{2})^{j-1}\,.
\]
\end{lemma}

\begin{proof}
It follows immediately from Definition \ref{Def.cot} and Lemma \ref{Lem.Psi.nZ}.
\end{proof}

In particular, when $n=1$, from Lemma \ref{Lem.cot.hn} we obtain
\begin{equation}\label{cot.dm.lim}
\lim_{s\to1/2}\frac{\ud^m}{\ud s^m}(\pi\cot(\pi s;A))
    =m!\{(-1)^m\hat{t}\a(m+1)-\t\a(m+1)\}=m!\check{t}\a(m+1)\,.
\end{equation}

Finally, we present the residue theorem given in Flajolet and Salvy's paper \cite{FlSa98}.

\begin{lemma}[\citu{FlSa98}{Lemma 2.1}]\label{Lem.Res}
Let $\xi(s)$ be a kernel function and let $r(s)$ be a rational function which is $O(s^{-2})$ at infinity. Then
\begin{equation}\label{Cau.Lind}
\sum_{\al\in O}{\rm Res}(r(s)\xi(s),\al)
    +\sum_{\be\in S}{\rm Res}(r(s)\xi(s),\be)=0\,,
\end{equation}
where $S$ is the set of poles of $r(s)$ and $O$ is the set of poles of $\xi(s)$ that are not poles of $r(s)$. Here ${\rm Res}(h(s),\la)$ denotes the residue of $h(s)$ at $s=\la$, and the kernel function $\xi(s)$ is meromorphic in the whole complex plane and satisfies $\xi(s)=o(s)$ over an infinite collection of circles $|s|=\rho_k$ with $\rho_k\to+\infty$.
\end{lemma}


\section{Linear sums}\label{Sec.lin}

Flajolet and Salvy \cite[Theorem 3.1]{FlSa98} applied the kernel function
\[
\frac{1}{2}\pi\cot(\pi s)\frac{\psi^{(p-1)}(-s)}{(p-1)!}
\]
to the base function $r(s)=s^{-q}$, and proved once again that all the linear Euler sums $S_{p,q}$ with an odd weight $w=p+q$ are reducible to polynomials in zeta values, which is a result extrapolated first by Euler without proof, and verified by Borwein et al. in \cite{BorBG95.PEM}.

Let $A=\{a_k\}_{k\in\mathbb{Z}}$ and $B=\{b_k\}_{k\in\mathbb{Z}}$ be real sequences satisfying $\lim_{k\to\infty}a_k/k=0$ and $\lim_{k\to\infty}b_k/k=0$. Now, we replace Flajolet and Salvy's function $\cot(\pi s)\psi^{(p-1)}(-s)$ by
\[
\cot(\pi s;A)\vPsi^{(p-1)}(\tf{1}{2}-s;B)\,,
\]
and use the residue computations to establish two general identities, which further yield the expressions of the linear (alternating) Euler $T$-sums and $\S$-sums.


\subsection{The first general identity and linear Euler $T$-sums}\label{Sec.lin.Tsum}

\begin{theorem}\label{Th.linear.sum1}
For integers $p\geq1$ and $q\geq2$, the following identity on sums related to the sequences $A$ and $B$ holds:
\begin{align*}
&(-1)^q\sum_{n=1}^\infty\frac{a_{n-1}\bar{N}\b_n(p)}{(n-1/2)^q}
    +(-1)^p\sum_{n=1}^\infty\frac{a_nN\b_n(p)}{(n-1/2)^q}
-\sum_{k=0}^{p-1}\binom{p+q-k-2}{q-1}\sum_{n=1}^\infty
    \frac{b_n S\a_{n+1}(k+1)}{n^{p+q-k-1}}\\
&\quad=-\sum_{k=1}^q\binom{p+k-2}{p-1}D\b(p+k-1)\check{t}\a(q-k+1)\\
&\quad\quad+b_0\{(-1)^q\hat{t}\a(p+q)+(-1)^p\t\a(p+q)\}\,.
\end{align*}
\end{theorem}

\begin{proof}
We consider the kernel function
\begin{equation}\label{kerfun.xi1}
\xi_1(s)=\pi\cot(\pi s;A)\frac{\vPsi^{(p-1)}(\tf{1}{2}-s;B)}{(p-1)!}
\end{equation}
and the base function $r_1(s)=(s-1/2)^{-q}$. Clearly, for the function $\mathcal{F}(s):=\xi_1(s)r_1(s)$, the only singularities are poles at the integers and half-integers $n-1/2$, where $n\in\mathbb{N}$. At an integer, the pole is simple. Then by Lemmas \ref{Lem.Psi.nZ} and \ref{Lem.cot.nZ}, the residue is
\[
{\rm Res}(\mathcal{F}(s),-n)=(-1)^{p+q}\frac{a_n\bar{N}\b_{n+1}(p)}{(n+1/2)^q}\,,\quad
    \text{for } n\geq 0\,,
\]
or
\[
{\rm Res}(\mathcal{F}(s),n)=\frac{a_nN\b_n(p)}{(n-1/2)^q}\,,\quad
    \text{for } n\geq 1\,.
\]
By Lemmas \ref{Lem.Psi.hn} and \ref{Lem.cot.hn}, the pole at a half-integer $n-1/2$ for $n\geq 2$ has order $p$, and the residue is
\[
{\rm Res}(\mathcal{F}(s),n-\tf{1}{2})
    =(-1)^{p-1}b_{n-1}\sum_{k=0}^{p-1}\binom{p+q-k-2}{q-1}\frac{S\a_{n}(k+1)}{(n-1)^{p+q-k-1}}\,.
\]
Next, according to Eq. (\ref{Psi.s.0.5}), the pole at $1/2$ has order $p+q$, and Eq. (\ref{cot.dm.lim}) shows that the residue is
\begin{align*}
{\rm Res}(\mathcal{F}(s),\tf{1}{2})
&=-b_0\{(-1)^{p+q}\hat{t}\a(p+q)+\t\a(p+q)\}\\
&\quad+(-1)^p\sum_{k=1}^q\binom{p+k-2}{p-1}D\b(p+k-1)\check{t}\a(q-k+1)\,.
\end{align*}
Summing these contributions yields the statement of the theorem.
\end{proof}

For simplicity, denote
\begin{equation}
\de_{\si}^k=1-\si(-1)^k\,,\quad\text{for }\si=\pm1\,.
\end{equation}
Thus, setting $A,B\in\{A_1,A_2\}$ in Theorem \ref{Th.linear.sum1} gives the following result.

\begin{theorem}\label{Th.lin.Tsum}
For integers $p\geq1$ and $q\geq2$, the linear Euler $T$-sums $T_{p,q}^{\si_1,\si_2}$ satisfy
\begin{align}
(-1)^p\de_{\si_1\si_2}^{p+q}T_{p,q}^{\si_1,\si_2}
&=(-1)^q\t(p+q;\si_1\si_2)-\si_2\de_{\si_2}^{q-1}\t(p;\si_1)\t(q;\si_2)\nonumber\\
&\quad+\sum_{k=1}^p\de_{\si_1\si_2}^{p-k}\binom{q+k-2}{q-1}
    \t(p-k+1;\si_1\si_2)\ze(q+k-1;\si_2)\nonumber\\
&\quad+\si_1\si_2\sum_{k=1}^q\de_{\si_1\si_2}^{q-k}\binom{p+k-2}{p-1}
    \t(q-k+1;\si_1\si_2)\ze(p+k-1;\si_1)\,,\label{Tpq}
\end{align}
where $\si_1,\si_2\in\{\pm1\}$, $\ze(1):=0$ and $\t(1):=2\ln(2)$.
\end{theorem}

\begin{proof}
Setting $A=A_1$ and $B=A_2$ in Theorem \ref{Th.linear.sum1}, substituting the evaluations of $\bar{N}_n^{(A_2)}(p)$, $N_n^{(A_2)}(p)$ and $S_{n+1}^{(A_1)}(k+1)$, etc. which are listed in Section \ref{Sec.Notation.Exp}, and doing some transformations, we obtain the result on the linear Euler $T$-sums $T_{\bar{p},\bar{q}}$:
\begin{align*}
&(-1)^p(1-(-1)^{p+q})\sum_{n=1}^\infty(-1)^{n-1}\frac{\bar{h}_{n-1}^{(p)}}{(n-1/2)^q}\\
&\quad=(-1)^q\t(p+q)+(1-(-1)^q)\t(\bar{p})\t(\bar{q})\\
&\quad\quad+\sum_{k=1}^p(1-(-1)^{p-k})\binom{q+k-2}{q-1}\t(p-k+1)\ze(\overline{q+k-1})\\
&\quad\quad+\sum_{k=1}^q(1-(-1)^{q-k})\binom{p+k-2}{p-1}\t(q-k+1)\ze(\overline{p+k-1})\,.
\end{align*}
Similarly, letting $(A,B)$ be $(A_1,A_1)$, $(A_2,A_1)$ and $(A_2,A_2)$ gives the results on the linear sums $T_{p,q}$, $T_{p,\bar{q}}$ and $T_{\bar{p},\bar{q}}$. All of the four cases can be covered by the general expression given in the theorem.
\end{proof}

\begin{remark}
Note that when $q=1$, the alternating linear sums $T_{p,\bar{1}}$ and $T_{\bar{p},\bar{1}}$ also satisfy the expression (\ref{Tpq}). We leave the details to the interested readers.
\end{remark}

By considering the parity of the parameters in Eq. (\ref{Tpq}), we obtain some specializations, which are showed in Corollaries \ref{Coro.Tpq.odd} and \ref{Coro.Tpq.even}.

\begin{corollary}\label{Coro.Tpq.odd}
For any odd weight $w=p+q$, the linear Euler $T$-sums $T_{p,q}$ are reducible to $\pi$, $\ln(2)$, and zeta values $\ze(3),\ze(5),\ldots$, while the linear Euler $T$-sums $T_{\bar{p},\bar{q}}$ are reducible to $\pi$, $\ln(2)$, $G$, zeta values $\ze(3),\ze(5),\ldots$, and the Dirichlet beta values $\be(4),\be(6),\ldots$.
\end{corollary}

\begin{example}
By further specifying the parameters, the following evaluations can be established:
\begin{align*}
&T_{1,2}=-\tf{7}{2}\ze(3)+\pi^2\ln(2)\,,\\
&T_{1,4}=-\tf{31}{2}\ze(5)-\tf{1}{2}\pi^2\ze(3)+\tf{1}{3}\pi^4\ln(2)\,,\\
&T_{2,3}=-\tf{31}{2}\ze(5)+\tf{3}{2}\pi^2\ze(3)\,,\\
&T_{3,2}=-\tf{31}{2}\ze(5)+2\pi^2\ze(3)\,,\\
&T_{1,6}=-\tf{127}{2}\ze(7)-\tf{1}{2}\pi^2\ze(5)-\tf{1}{6}\pi^4\ze(3)
    +\tf{2}{15}\pi^6\ln(2)\,,\\
&T_{2,5}=-\tf{127}{2}\ze(7)+\tf{5}{2}\pi^2\ze(5)+\tf{1}{3}\pi^4\ze(3)\,,\\
&T_{3,4}=-\tf{127}{2}\ze(7)-5\pi^2\ze(5)+\pi^4\ze(3)\,,\\
&T_{4,3}=-\tf{127}{2}\ze(7)+5\pi^2\ze(5)+\tf{1}{6}\pi^4\ze(3)\,,\\
&T_{5,2}=-\tf{127}{2}\ze(7)+13\pi^2\ze(5)-\tf{1}{3}\pi^4\ze(3)\,.
\end{align*}
Note that the sum $T_{1,2}$ can be found in \cite[Eq. (33)]{Chen06}.\hfill\qedsymbol
\end{example}

\begin{example}
Similarly, we have
\begin{align*}
&T_{\bar{1},\bar{2}}=-\tf{7}{2}\ze(3)+\tf{1}{2}\pi^2\ln(2)\,,\\
&T_{\bar{2},\bar{1}}=-\tf{7}{2}\ze(3)+2\pi G-\tf{1}{2}\pi^2\ln(2)\,,\\
&T_{\bar{1},\bar{4}}=-\tf{31}{2}\ze(5)+\tf{3}{8}\pi^2\ze(3)+\tf{1}{6}\pi^4\ln(2)\,,\\
&T_{\bar{2},\bar{3}}=-\tf{31}{2}\ze(5)-\tf{9}{8}\pi^2\ze(3)+\pi^3G\,,\\
&T_{\bar{3},\bar{2}}=-\tf{31}{2}\ze(5)+\tf{9}{8}\pi^2\ze(3)\,,\\
&T_{\bar{4},\bar{1}}=-\tf{31}{2}\ze(5)+8\pi\be(4)-\tf{3}{8}\pi^2\ze(3)
    -\tf{1}{6}\pi^4\ln(2)\,,\\
&T_{\bar{1},\bar{6}}=-\tf{127}{2}\ze(7)+\tf{15}{32}\pi^2\ze(5)
    +\tf{1}{8}\pi^4\ze(3)+\tf{1}{15}\pi^6\ln(2)\,,\\
&T_{\bar{2},\bar{5}}=-\tf{127}{2}\ze(7)-\tf{75}{32}\pi^2\ze(5)-\tf{1}{4}\pi^4\ze(3)
    +\tf{5}{12}\pi^5G\,,\\
&T_{\bar{3},\bar{4}}=-\tf{127}{2}\ze(7)+\tf{75}{16}\pi^2\ze(5)
    +\tf{1}{8}\pi^4\ze(3)\,,\\
&T_{\bar{4},\bar{3}}=-\tf{127}{2}\ze(7)-\tf{75}{16}\pi^2\ze(5)
    +4\pi^3\be(4)-\tf{1}{8}\pi^4\ze(3)\,,\\
&T_{\bar{5},\bar{2}}=-\tf{127}{2}\ze(7)+\tf{75}{32}\pi^2\ze(5)+\tf{1}{4}\pi^4\ze(3)\,,\\
&T_{\bar{6},\bar{1}}=-\tf{127}{2}\ze(7)+32\pi\be(6)-\tf{15}{32}\pi^2\ze(5)
    -\tf{1}{8}\pi^4\ze(3)-\tf{1}{15}\pi^6\ln(2)\,.
\end{align*}
\end{example}

\begin{corollary}\label{Coro.Tpq.even}
For any even weight $w=p+q$, the linear Euler $T$-sums $T_{p,\bar{q}}$ and $T_{\bar{p},q}$ are reducible to $\pi$, $\ln(2)$, $G$, zeta values $\ze(3),\ze(5),\ldots$, and the Dirichlet beta values $\be(4),\be(6),\ldots$.
\end{corollary}

\begin{example}
For the linear Euler $T$-sums $T_{p,\bar{q}}$, we have
\begin{align*}
&T_{1,\bar{1}}=-2G+\tf{1}{2}\pi\ln(2)\,,\\
&T_{1,\bar{3}}=-8\be(4)+\tf{1}{2}\pi^3\ln(2)-\tf{7}{8}\pi\ze(3)\,,\\
&T_{2,\bar{2}}=-8\be(4)+\tf{7}{4}\pi\ze(3)\,,\\
&T_{3,\bar{1}}=-8\be(4)+\tf{21}{8}\pi\ze(3)-\tf{1}{4}\pi^3\ln(2)\,,\\
&T_{1,\bar{5}}=-32\be(6)-\tf{31}{32}\pi\ze(5)-\tf{1}{4}\pi^3\ze(3)
    +\tf{5}{24}\pi^5\ln(2)\,,\\
&T_{2,\bar{4}}=-32\be(6)+\tf{31}{8}\pi\ze(5)+\tf{1}{2}\pi^3\ze(3)\,,\\
&T_{3,\bar{3}}=-32\be(6)-\tf{93}{16}\pi\ze(5)+\tf{21}{16}\pi^3\ze(3)\,,\\
&T_{4,\bar{2}}=-32\be(6)+\tf{31}{8}\pi\ze(5)+\tf{3}{8}\pi^3\ze(3)\,,\\
&T_{5,\bar{1}}=-32\be(6)+\tf{465}{32}\pi\ze(5)-\tf{3}{16}\pi^3\ze(3)-\tf{5}{48}\pi^5\ln(2)\,.
\end{align*}
Note that the evaluation of $T_{1,\bar{1}}$ gives the corrected version of \cite[Eq. (4.5b)]{Chu97.HSR}:
\[
\sum_{n=0}^{\infty}(-1)^n\frac{O_n}{2n+1}=\frac{1}{4}T_{1,\bar{1}}
    =-\frac{1}{2}G+\frac{1}{8}\pi\ln(2)\,,
\]
where $O_n^{(r)}$ are also called odd harmonic numbers in the literature, defined by $O_0^{(r)}=0$, $O_n^{(r)}=\sum_{k=1}^n\frac{1}{(2k-1)^r}$, for $n,r=1,2,\ldots$, and satisfy
\[
O_n^{(p)}=H_{2n}^{(p)}-\frac{1}{2^p}H_n^{(p)}=\frac{1}{2^p}h_n^{(p)}\,.\hfill\qedsymbol
\]
\end{example}

\begin{example}
For the linear Euler $T$-sums $T_{\bar{p},q}$, we have
\begin{align*}
&T_{\bar{1},3}=-8\be(4)+\tf{7}{8}\pi\ze(3)+\tf{1}{4}\pi^3\ln(2)\,,\\
&T_{\bar{2},2}=-8\be(4)-\tf{7}{4}\pi\ze(3)+2\pi^2G\,,\\
&T_{\bar{1},5}=-32\be(6)+\tf{31}{32}\pi\ze(5)+\tf{3}{16}\pi^3\ze(3)+\tf{5}{48}\pi^5\ln(2)\,,\\
&T_{\bar{2},4}=-32\be(6)-\tf{31}{8}\pi\ze(5)-\tf{3}{8}\pi^3\ze(3)+\tf{2}{3}\pi^4G\,,\\
&T_{\bar{3},3}=-32\be(6)+\tf{93}{16}\pi\ze(5)+\tf{7}{16}\pi^3\ze(3)\,,\\
&T_{\bar{4},2}=-32\be(6)-\tf{31}{8}\pi\ze(5)+8\pi^2\be(4)-\tf{1}{2}\pi^3\ze(3)\,.
\end{align*}
\end{example}

Next, let us give the parity theorem of the Hoffman double $t$-values.

\begin{theorem}
When the weight $w=p+q$ and the quantity $\frac{1+\si_1\si_2}{2}$ have the same parity, the double $t$-values $t(p,q;\si_1,\si_2)$ and double $\t$-values $\t(p,q;\si_1,\si_2)$ are reducible to combinations of $\pi$, $\ln(2)$, $G$, zeta values $\ze(3),\ze(5),\ldots$, and the Dirichlet beta values $\be(4),\be(6),\ldots$.
\end{theorem}

\begin{proof}
It can be found that
\begin{equation}\label{Tpq.DMtV}
T_{p,q}^{\si_1,\si_2}
=\sum_{n=1}^{\infty}\si_2^{n-1}\frac{h_{n-1}^{(p)}(\si_1)}{(n-1/2)^q}
=\sum_{n=1}^{\infty}\frac{\si_2^{n-1}}{(n-1/2)^q}
    \sum_{k=1}^{n-1}\frac{\si_1^{k-1}}{(k-1/2)^p}
=\si_1\si_2\t(q,p;\si_2,\si_1)\,.
\end{equation}
Thus, the assertion of this theorem holds according to Corollaries \ref{Coro.Tpq.odd} and \ref{Coro.Tpq.even}.
\end{proof}

\begin{example}
Based on Eq. (\ref{Tpq.DMtV}) and Hoffman's result \cite[Eq. (4.4)]{Hoff19.AOV}, the linear $T$-sums $T_{p,q}$ can be expressed in terms as (alternating) MZVs:
\begin{equation}
T_{p,q}=\t(q,p)=2^{p+q-2}\{\ze(q,p)-\ze(q,\bar{p})-\ze(\bar{q},p)+\ze(\bar{q},\bar{p})\}\,,
\end{equation}
from which, we further obtain
\begin{align*}
&T_{1,3}=-16\Lii_4(\tf{1}{2})-\tf{2}{3}\ln(2)^4+\tf{2}{3}\pi^2\ln(2)^2+\tf{23}{360}\pi^4\,,\\
&T_{2,2}=\tf{1}{24}\pi^4\,,\\
&T_{1,5}=-32\ze(\bar{5},1)+62\ln(2)\ze(5)+\tf{17}{2}\ze(3)^2-\tf{73}{1260}\pi^6\,,\\
&T_{2,4}=7\ze(3)^2-\tf{1}{105}\pi^6\,,\\
&T_{3,3}=\tf{49}{2}\ze(3)^2-\tf{1}{30}\pi^6\,,\\
&T_{4,2}=\tf{11}{420}\pi^6-7\ze(3)^2\,.
\end{align*}
Additionally, based on the values of $t(\bar{1},\bar{1})$, $t(\bar{2},\bar{2})$ and $t(\bar{3},\bar{3})$ given in \cite[Section 6]{Hoff19.AOV}, the evaluations of the following sums can be established:
\begin{align*}
&T_{\bar{1},\bar{1}}=-\tf{1}{8}\pi^2\,,\\
&T_{\bar{2},\bar{2}}=-\tf{1}{12}\pi^4+8G^2\,,\\
&T_{\bar{3},\bar{3}}=-\tf{1}{480}\pi^6\,.
\end{align*}
More linear $T$-sums, which can not be covered by Theorem \ref{Th.lin.Tsum}, will be presented in Section \ref{Sec.linES.CMZV} by colored multiple zeta values.\hfill\qedsymbol
\end{example}


\subsection{The second general identity and linear Euler $\S$-sums}\label{Sec.lin.Stsum}

By applying the same kernel function to a different base function, another general identity can be established.

\begin{theorem}\label{Th.linear.sum2}
For integers $p\geq1$ and $q\geq2$, the following identity on sums related to the sequences $A$ and $B$ holds:
\begin{align*}
&(-1)^q\sum_{n=1}^\infty\frac{a_n\bar{N}\b_{n+1}(p)}{n^q}
    +(-1)^p\sum_{n=1}^\infty\frac{a_nN\b_n(p)}{n^q}
    -\sum_{k=0}^{p-1}\binom{p+q-k-2}{q-1}
        \sum_{n=1}^\infty\frac{b_{n-1}S\a_n(k+1)}{(n-1/2)^{p+q-k-1}}\\
&\quad=2\sum_{j=1}^{[q/2]}\binom{p+q-2j-1}{p-1}D\a(2j)\hat{t}\b(p+q-2j)
    -a_0\binom{p+q-1}{p-1}\hat{t}\b(p+q)\,.
\end{align*}
\end{theorem}

\begin{proof}
The proof is similar to that of Theorem \ref{Th.linear.sum1}. In this case, apply the kernel function $\xi_1(s)$ in (\ref{kerfun.xi1}) to the base function $r_2(s)=s^{-q}$, and use the expansions in Lemmas \ref{Lem.Psi.nZ}--\ref{Lem.cot.hn} and Eqs. (\ref{Psi.s.0}) and (\ref{cot.s.0}) to perform the residue computation.
\end{proof}

Now, setting $A,B$ by $A_1$ or $A_2$, we obtain from Theorem \ref{Th.linear.sum2} the following result.

\begin{theorem}\label{Th.lin.Stsum}
For integers $p\geq1$ and $q\geq2$, the linear Euler $\S$-sums $\S_{p,q}^{\si_1,\si_2}$ satisfy
\begin{align}
\si_2(-1)^p\de_{\si_1}^{p+q}\S_{p,q}^{\si_1,\si_2}
&=-(1+(-1)^q)\t(p;\si_1)\ze(q;\si_2)-\binom{p+q-1}{p-1}\t(p+q;\si_1)\nonumber\\
&\quad+\sum_{k=1}^p\de_{\si_1\si_2}^{k-1}\binom{p+q-k-1}{q-1}
    \t(p+q-k;\si_2)\t(k;\si_1\si_2)\nonumber\\
&\quad+2\sum_{k=1}^{[\frac{q}{2}]}\binom{p+q-2j-1}{p-1}
    \t(p+q-2j;\si_1)\ze(2j;\si_1\si_2)\,,\label{lin.Stsum}
\end{align}
where $\si_1,\si_2\in\{\pm1\}$, $\ze(1):=0$ and $\t(1):=2\ln(2)$.
\end{theorem}

It can be verified that when $q=1$, the expression given in Theorem \ref{Th.lin.Stsum} also holds for the alternating sums $\S_{p,\bar{1}}$ and $\S_{\bar{p},\bar{1}}$. Moreover, we have

\begin{corollary}\label{Coro.Stpq.odd}
For any odd weight $w=p+q$, the linear Euler $\S$-sums $\S_{p,q}$ are reducible to $\pi$ and zeta values $\ze(3),\ze(5),\ldots$, while the linear Euler $\S$-sums $\S_{p,\bar{q}}$ are reducible to $\pi$, $G$, zeta values $\ze(3),\ze(5),\ldots$, and the Dirichlet beta values $\be(4),\be(6),\ldots$.
\end{corollary}

\begin{proof}
This corollary gives two specializations of Eq. (\ref{lin.Stsum}). Note that when the weight $w=p+q$ is odd, the combinations of the terms on $\t(1)$ and $\ze(\bar{1})$ in the expressions of $\S_{p,q}$ and $\S_{p,\bar{q}}$ are vanish. Therefore, the evaluations of $\S_{p,q}$ and $\S_{p,\bar{q}}$ do not involve $\ln(2)$.
\end{proof}

\begin{example}
For linear sums $\S_{p,q}$, we have
\begin{align*}
&\S_{1,2}=\tf{7}{2}\ze(3)\,,\\
&\S_{1,4}=\tf{31}{2}\ze(5)-\tf{7}{6}\pi^2\ze(3)\,,\\
&\S_{2,3}=-62\ze(5)+\tf{35}{6}\pi^2\ze(3)\,,\\
&\S_{3,2}=93\ze(5)-7\pi^2\ze(3)\,,\\
&\S_{1,6}=\tf{127}{2}\ze(7)-\tf{31}{6}\pi^2\ze(5)-\tf{7}{90}\pi^4\ze(3)\,,\\
&\S_{2,5}=-381\ze(7)+\tf{217}{6}\pi^2\ze(5)+\tf{7}{45}\pi^4\ze(3)\,,\\
&\S_{3,4}=\tf{1905}{2}\ze(7)-93\pi^2\ze(5)\,,\\
&\S_{4,3}=-1270\ze(7)+\tf{341}{3}\pi^2\ze(5)+\tf{7}{6}\pi^4\ze(3)\,,\\
&\S_{5,2}=\tf{1905}{2}\ze(7)-62\pi^2\ze(5)-\tf{7}{3}\pi^4\ze(3)\,.
\end{align*}
Note that the evaluation of $\S_{1,2}$ is given in, e.g., \cite[Chapter 9, Entry 12]{Berndt85.1}, \cite[p. 1198]{BorBor95.OAI}, \cite[Eq. (3.2)]{Chu97.HSR} and \cite[Eq. (15)]{DeDo91}, and the evaluation of $\S_{1,4}$ can be found in \cite[Eq. (3.7a)]{ZhengDY07}.
\hfill\qedsymbol
\end{example}

\begin{example}
Similarly, for linear sums $\S_{p,\bar{q}}$, we have
\begin{align*}
&\S_{1,\bar{2}}=-\tf{7}{2}\ze(3)+2\pi G\,,\\
&\S_{2,\bar{1}}=7\ze(3)-2\pi G\,,\\
&\S_{1,\bar{4}}=-\tf{31}{2}\ze(5)+8\pi\be(4)-\tf{7}{12}\pi^2\ze(3)\,,\\
&\S_{2,\bar{3}}=62\ze(5)-24\pi\be(4)+\tf{7}{6}\pi^2\ze(3)\,,\\
&\S_{3,\bar{2}}=-93\ze(5)+24\pi\be(4)+\pi^3G\,,\\
&\S_{4,\bar{1}}=62\ze(5)-8\pi\be(4)-\pi^3G\,,\\
&\S_{1,\bar{6}}=-\tf{127}{2}\ze(7)+32\pi\be(6)-\tf{31}{12}\pi^2\ze(5)
    -\tf{49}{720}\pi^4\ze(3)\,,\\
&\S_{2,\bar{5}}=381\ze(7)-160\pi\be(6)+\tf{31}{3}\pi^2\ze(5)+\tf{49}{360}\pi^4\ze(3)\,,\\
&\S_{3,\bar{4}}=-\tf{1905}{2}\ze(7)+320\pi\be(6)-\tf{31}{2}\pi^2\ze(5)+4\pi^3\be(4)\,,\\
&\S_{4,\bar{3}}=1270\ze(7)-320\pi\be(6)+\tf{31}{3}\pi^2\ze(5)-12\pi^3\be(4)\,,\\
&\S_{5,\bar{2}}=-\tf{1905}{2}\ze(7)+160\pi\be(6)+12\pi^3\be(4)+\tf{5}{12}\pi^5G\,,\\
&\S_{6,\bar{1}}=381\ze(7)-32\pi\be(6)-4\pi^3\be(4)-\tf{5}{12}\pi^5G\,.
\end{align*}
The evaluation of $\S_{1,\bar{2}}$ is presented in \cite[Chapter 9, Entry 21]{Berndt85.1}, \cite[Eq. (28)]{Chen06} and \cite[Eq. (19)]{DeDo91} in slightly different forms, and the evaluation of $\S_{2,\bar{1}}$ is given in \cite[Example 6]{CopCan15}.
\hfill\qedsymbol
\end{example}

Furthermore, by transformation, we have
\begin{align*}
\S_{p,q}&=\sum_{n=1}^{\infty}\frac{h_n^{(p)}}{n^q}
    =2^{p+q}\sum_{n=1}^{\infty}\frac{H_{2n}^{(p)}}{(2n)^q}
        -\sum_{n=1}^{\infty}\frac{H_n^{(p)}}{n^q}
    =2^{p+q}\sum_{n=1}^{\infty}\frac{1+(-1)^n}{2}\frac{H_n^{(p)}}{n^q}
        -\sum_{n=1}^{\infty}\frac{H_n^{(p)}}{n^q}\\
&=(2^{p+q-1}-1)S_{p,q}-2^{p+q-1}S_{p,\bar{q}}\,.
\end{align*}
According to Flajolet and Salvy's results \cite[Theorems 7.1 and 7.2]{FlSa98} (see also \cite[Corollary 5.1]{XuWang19.EFES}), the linear Euler $\S$-sums $\S_{p,q}$ are expressible in terms of (alternating) zeta values and double zeta values. In particular, the following expression holds:
\begin{equation}\label{St.pq}
\S_{p,q}=(2^{p+q-1}-1)\{\ze(q,p)+\ze(p+q)\}
    +2^{p+q-1}\{\ze(\bar{q},p)+\ze(\ol{p+q})\}\,.
\end{equation}

\begin{example}
Specifying the parameters in (\ref{St.pq}) yields
\begin{align*}
&\S_{1,3}=16\Lii_4(\tf{1}{2})+14\ln(2)\ze(3)
    +\tf{2}{3}\ln(2)^4-\tf{2}{3}\pi^2\ln(2)^2-\tf{53}{360}\pi^4\,,\\
&\S_{2,2}=-32\Lii_4(\tf{1}{2})-28\ln(2)\ze(3)
    -\tf{4}{3}\ln(2)^4+\tf{4}{3}\pi^2\ln(2)^2+\tf{151}{360}\pi^4\,,\\
&\S_{1,5}=32\ze(\bar{5},1)-\tf{31}{2}\ze(3)^2+\tf{31}{1260}\pi^6\,,\\
&\S_{2,4}=-128\ze(\bar{5},1)+55\ze(3)^2-\tf{47}{630}\pi^6\,,\\
&\S_{3,3}=192\ze(\bar{5},1)-\tf{95}{2}\ze(3)^2+\tf{8}{105}\pi^6\,,\\
&\S_{4,2}=-128\ze(\bar{5},1)-\ze(3)^2+\tf{41}{1260}\pi^6\,.
\end{align*}
These sums can not be obtained from Theorem \ref{Th.lin.Stsum}. Note that
\[
G(1)=\sum_{n=1}^{\infty}\frac{O_n}{(2n)^3}=\frac{1}{16}\S_{1,3}
\]
is related to an incorrect formula claimed by Ramanujan. The readers are referred to Berndt's remark \cite[p. 257]{Berndt85.1} for the history of this constant, and to the papers due to Sitaramachandra Rao \cite{Sit87}, and Coppo and Candelpergher \cite[Section 6]{CopCan15} for more details and identities on it.\hfill\qedsymbol
\end{example}

Contrary to Corollary \ref{Coro.Stpq.odd}, for even weights, Theorem \ref{Th.lin.Stsum} yields the next result.

\begin{corollary}\label{Coro.Stpq.even}
For any even weight $w=p+q$, the linear Euler $\S$-sums $\S_{\bar{p},q}$ are reducible to $\pi$, $G$, zeta values $\ze(3),\ze(5),\ldots$, and the Dirichlet beta values $\be(4),\be(6),\ldots$, while the linear Euler $\S$-sums $\S_{\bar{p},\bar{q}}$ are reducible to $\pi$, $G$, and the Dirichlet beta values $\be(4),\be(6),\ldots$.
\end{corollary}

\begin{example}
For linear sums $\S_{\bar{p},q}$, we have
\begin{align*}
&\S_{\bar{1},3}=-8\be(4)+\tf{7}{2}\pi\ze(3)-\tf{1}{3}\pi^2G\,,\\
&\S_{\bar{2},2}=24\be(4)-7\pi\ze(3)+\pi^2G\,,\\
&\S_{\bar{1},5}=-32\be(6)+\tf{31}{2}\pi\ze(5)-\tf{4}{3}\pi^2\be(4)-\tf{7}{180}\pi^4G\,,\\
&\S_{\bar{2},4}=160\be(6)-62\pi\ze(5)+4\pi^2\be(4)+\tf{1}{12}\pi^4G\,,\\
&\S_{\bar{3},3}=-320\be(6)+93\pi\ze(5)-4\pi^2\be(4)+\tf{7}{4}\pi^3\ze(3)\,,\\
&\S_{\bar{4},2}=320\be(6)-62\pi\ze(5)+4\pi^2\be(4)-\tf{7}{2}\pi^3\ze(3)\,.
\end{align*}
\end{example}

\begin{example}
For linear sums $\S_{\bar{p},\bar{q}}$, we have
\begin{align*}
&\S_{\bar{1},\bar{1}}=2G\,,\\
&\S_{\bar{1},\bar{3}}=8\be(4)-\tf{2}{3}\pi^2G\,,\\
&\S_{\bar{2},\bar{2}}=-24\be(4)+3\pi^2G\,,\\
&\S_{\bar{3},\bar{1}}=24\be(4)-2\pi^2G\,,\\
&\S_{\bar{1},\bar{5}}=32\be(6)-\tf{8}{3}\pi^2\be(4)-\tf{2}{45}\pi^4G\,,\\
&\S_{\bar{2},\bar{4}}=-160\be(6)+16\pi^2\be(4)+\tf{1}{12}\pi^4G\,,\\
&\S_{\bar{3},\bar{3}}=320\be(6)-32\pi^2\be(4)\,,\\
&\S_{\bar{4},\bar{2}}=-320\be(6)+28\pi^2\be(4)+\tf{2}{3}\pi^4G\,,\\
&\S_{\bar{5},\bar{1}}=160\be(6)-8\pi^2\be(4)-\tf{2}{3}\pi^4G\,.
\end{align*}
Some linear $\S$-sums, which can not be deduced from Theorem \ref{Th.lin.Stsum}, will be evaluated in Section \ref{Sec.linES.CMZV} by colored multiple zeta values.\hfill\qedsymbol
\end{example}

Now, let us give the parity theorem of the Kaneko-Tsumura double $T$-values.

\begin{theorem}
When the weight $w=p+q$ and the quantity $\frac{1+\si_2}{2}$ have the same parity, the double $T$-values $T(p,q;\si_1,\si_2)$ are reducible to combinations of $\pi$, $G$, zeta values $\ze(3),\ze(5),\ldots$, and the Dirichlet beta values $\be(4),\be(6),\ldots$.
\end{theorem}

\begin{proof}
It follows from
\begin{align}
\S_{p,q}^{\si_1,\si_2}
    &=\sum_{n=1}^\infty\si_2^{n-1}\frac{h_n^{(p)}(\si_1)}{n^q}
        =\sum_{n=1}^\infty\frac{\si_2^{n-1}}{n^q}
            \sum_{k=1}^n\frac{\si_1^{k-1}}{(k-1/2)^p}\nonumber\\
    &=2^{p+q-2}\si_1\cdot2^2\sum_{n>k\geq 1}\frac{\si_2^n\si_1^k}{(2n-2)^q(2k-1)^p}
        =2^{p+q-2}\si_1T(q,p;\si_2,\si_1)\label{Stpq.DMTV}
\end{align}
and Corollaries \ref{Coro.Stpq.odd} and \ref{Coro.Stpq.even}.
\end{proof}


\subsection{Linear Euler sums and colored MZVs}\label{Sec.linES.CMZV}

Now, define
\[
\Li_{s_1,s_2,\ldots,s_k}(\si_1,\si_2,\ldots,\si_k)
    =\sum_{n_1>n_2>\cdots>n_k\geq 1}
        \frac{\si_1^{n_1}\si_2^{n_2}\cdots\si_k^{n_k}}
            {n_1^{s_1}n_2^{s_2}\cdots n_k^{s_k}}\,,
\]
where $\si_j$ are the $N$th roots of unity, and $s_j$ are positive integers, for $j=1,2,\ldots,k$, with $(s_1,\si_1)\neq(1,1)$. They are in fact \emph{multiple polylogarithm values} at roots of unity, also called \emph{colored multiple zeta values} (CMZVs) of \emph{level} $N$ (see, for example, \cite{BJOP2002,Zhao08.MPV}). Since $\si_j=\om^{r_j}$ for $r_j=0,1,\ldots,N-1$ and $\om=\exp(2\pi\ui/N)$, for convenience, we use the auxiliary notation
\[
L_N(s_1,s_2,\ldots,s_k;r_1,r_2,\ldots,r_k)
    =\Li_{s_1,s_2,\ldots,s_k}(\om^{r_1},\om^{r_2},\ldots,\om^{r_k})
\]
to denote CMZVs of lever $N$. Then, the following theorem can be deduced.

\begin{theorem}\label{Th.TS.CMZV}
The linear Euler $T$-sums $T_{p,q}^{\si_1,\si_2}$ and $\S$-sums $\S_{p,q}^{\si_1,\si_2}$ are expressible in terms of CMZVs of weight $p+q$ and of level four:
\[
T_{p,q}^{\si_1,\si_2}=2^{p+q-2}\si_1^{\frac{3}{2}}\si_2^{\frac{3}{2}}
    \bibb{
        \begin{array}{c}
            L_4(q,p;\la_{\si_2},\la_{\si_1})-L_4(q,p;\la_{\si_2},\la_{\si_1}+2)\\
            -L_4(q,p;\la_{\si_2}+2,\la_{\si_1})+L_4(q,p;\la_{\si_2}+2,\la_{\si_1}+2)
        \end{array}
    }\,,
\]
and
\[
\S_{p,q}^{\si_1,\si_2}=2^{p+q-2}\si_1^{\frac{3}{2}}\si_2
    \bibb{
        \begin{array}{c}
            L_4(q,p;\la_{\si_2},\la_{\si_1})-L_4(q,p;\la_{\si_2},\la_{\si_1}+2)\\
            +L_4(q,p;\la_{\si_2}+2,\la_{\si_1})-L_4(q,p;\la_{\si_2}+2,\la_{\si_1}+2)
        \end{array}
    }\,,
\]
where $\la_{\si}=0$ if $\si=1$, and $\la_{\si}=1$ if $\si=-1$.
\end{theorem}

\begin{proof}
Note that $\si=\exp(\la_{\si}\pi\ui)$ for $\si=\pm1$. Then according to the definitions, we have
\begin{align*}
T_{p,q}^{\si_1,\si_2}
&=\sum_{n>k\geq 1}\frac{\si_2^{n-1}\si_1^{k-1}}{(n-1/2)^q(k-1/2)^p}
    =2^{p+q}\si_1\si_2\sum_{n>k\geq 1}\frac{\si_2^n\si_1^k}{(2n-1)^q(2k-1)^p}\\
&=2^{p+q-2}\si_1\si_2\sum_{m_1>m_2\geq1}
    \frac{\si_2^{(m_1+1)/2}\si_1^{(m_2+1)/2}(1-(-1)^{m_1})(1-(-1)^{m_2})}{m_1^qm_2^p}\\
&=2^{p+q-2}\si_1^{\frac{3}{2}}\si_2^{\frac{3}{2}}
    \sum_{m_1>m_2\geq1}\frac{\om^{\la_{\si_2}m_1}\om^{\la_{\si_1}m_2}
        (1-\om^{2m_1})(1-\om^{2m_2})}{m_1^qm_2^p}\,,
\end{align*}
where $\om=\exp(\pi\ui/2)=\ui$. Expanding the above formula and using the definition of CMZVs of level four, we obtain the expression of linear sums $T_{p,q}^{\si_1,\si_2}$. The expression of linear sums $\S_{p,q}^{\si_1,\si_2}$ can be established in a similar way.
\end{proof}

By using the above expressions and the CMZVs computation function in the Mathematica package \texttt{MultipleZetaValues} developed by Au \cite{Au20.ASD} most recently, we obtain the evaluations of more linear $T$-sums and $\S$-sums, which can not been obtained from the results in Sections \ref{Sec.lin.Tsum} and \ref{Sec.lin.Stsum}.

\begin{example}
The following evaluations of linear Euler $T$-sums hold:
\begin{align*}
&T_{\bar{1},2}=-8\imm(\Li_3(\tf{1}{2}+\tf{\ui}{2}))
    +\tf{1}{4}\pi\ln(2)^2+\tf{3}{16}\pi^3\,,\\
&T_{1,\bar{2}}=-16\imm(\Li_3(\tf{1}{2}+\tf{\ui}{2}))
    +\tf{1}{2}\pi\ln(2)^2+\tf{1}{4}\pi^3\,,\\
&T_{2,\bar{1}}=8\imm(\Li_3(\tf{1}{2}+\tf{\ui}{2}))
    -\tf{1}{4}\pi\ln(2)^2-\tf{3}{16}\pi^3\,,\\
&T_{\bar{1},\bar{3}}=-8\Lii_4(\tf{1}{2})-\tf{1}{3}\ln(2)^4
    +\tf{1}{3}\pi^2\ln(2)^2+\tf{31}{1440}\pi^4\,,\\
&T_{\bar{3},\bar{1}}=8\Lii_4(\tf{1}{2})+\tf{1}{3}\ln(2)^4
    -\tf{1}{3}\pi^2\ln(2)^2-\tf{91}{1440}\pi^4\,.
\end{align*}
\end{example}

\begin{example}
The following evaluations of linear Euler $\S$-sums hold:
\begin{align*}
&\S_{1,\bar{1}}=\tf{1}{8}\pi^2\,,\\
&\S_{1,\bar{3}}=-8\Lii_4(\tf{1}{2})-7\ln(2)\ze(3)
    -\tf{1}{3}\ln(2)^4+\tf{1}{3}\pi^2\ln(2)^2+\tf{151}{1440}\pi^4\,,\\
&\S_{2,\bar{2}}=16\Lii_4(\tf{1}{2})+14\ln(2)\ze(3)+8G^2+\tf{2}{3}\ln(2)^4
    -\tf{2}{3}\pi^2\ln(2)^2-\tf{151}{720}\pi^4\,,\\
&\S_{3,\bar{1}}=\tf{1}{8}\pi^4-8G^2\,,
\end{align*}
and
\begin{align*}
&\S_{\bar{1},2}=-8\imm(\Li_3(\tf{1}{2}+\tf{\ui}{2}))-4\ln(2)G
    +\tf{1}{4}\pi\ln(2)^2+\tf{5}{16}\pi^3\,,\\
&\S_{\bar{1},\bar{2}}=16\imm(\Li_3(\tf{1}{2}+\tf{\ui}{2}))+8\ln(2)G
    -\tf{1}{2}\pi\ln(2)^2-\tf{3}{8}\pi^3\,,\\
&\S_{\bar{2},\bar{1}}=-8\imm(\Li_3(\tf{1}{2}+\tf{\ui}{2}))
    -4\ln(2)G+\tf{1}{4}\pi\ln(2)^2+\tf{5}{16}\pi^3\,,
\end{align*}
where the values of $\S_{1,\bar{1}}$ and $\S_{3,\bar{1}}$ are presented in \cite[Example 6]{CopCan15}. See also \cite[Eq. (24)]{Chen06} and \cite[Eq. (4.1b)]{Chu97.HSR} for $\S_{1,\bar{1}}$.\hfill\qedsymbol
\end{example}


\subsection{Linear Euler $R$-sums}\label{Sec.lin.Rsum}

The second named author \cite{X2020.EFS} also considered another variant of the linear Euler sums:
\[
R_{p,q}^{\si_1,\si_2}=\sum_{n=1}^{\infty}\si_2^{n-1}\frac{H_{n-1}^{(p)}(\si_1)}{(n-1/2)^q}\,,
\]
where the numbers $H_n^{(p)}(\si_1)$ represent the classical harmonic numbers $H_n^{(p)}$ for $\si_1=1$ and the alternating harmonic numbers $\bar{H}_n^{(p)}$ for $\si_1=-1$. This kind of Euler sums will be used in the computation of quadratic Euler $\S$-sums in Section \ref{Sec.qua.Stsum}.

The results given in \cite{X2020.EFS} can be reformulated by the following unified identity:
\begin{align}
\si_1\si_2(-1)^p\de_{\si_2}^{p+q}R_{p,q}^{\si_1,\si_2}
    &=-\de_{\si_2}^{q-1}\ze(p;\si_1)\t(q;\si_2)\nonumber\\
&\quad+\sum_{k=0}^p(1+(-1)^k)\binom{p+q-k-1}{q-1}\ze(k;\si_1\si_2)\t(p+q-k;\si_2)\nonumber\\
&\quad+\sum_{k=1}^q\de_{\si_1\si_2}^{k-1}\binom{p+q-k-1}{p-1}
    \t(k;\si_1\si_2)\t(p+q-k;\si_1)\,,\label{Rpq}
\end{align}
where $\ze(0)$ should be interpreted as $-1/2$ wherever it occurs. Moreover, setting $a=-1/2$ in \cite[Eq. (1.8)]{Xu17.SEP} yields
\begin{equation}\label{R1q}
R_{1,q}=\frac{q}{2}\t(q+1)-\frac{1}{2}\sum_{k=0}^{q-1}\t(q-k)\t(k+1)\,.
\end{equation}
Similarly to Theorem \ref{Th.TS.CMZV}, based on the CMZVs, we obtain the next result.

\begin{theorem}\label{Th.Rpq.CMZV}
The linear Euler $R$-sums $R_{p,q}^{\si_1,\si_2}$ are expressible in terms of CMZVs of weight $p+q$ and of level four:
\[
R_{p,q}^{\si_1,\si_2}=2^{p+q-2}\si_1\si_2^{\frac{3}{2}}
    \bibb{
        \begin{array}{c}
            L_4(q,p;\la_{\si_2},\la_{\si_1})+L_4(q,p;\la_{\si_2},\la_{\si_1}+2)\\
            -L_4(q,p;\la_{\si_2}+2,\la_{\si_1})-L_4(q,p;\la_{\si_2}+2,\la_{\si_1}+2)
        \end{array}
    }\,,
\]
where $\la_{\si}=0$ if $\si=1$, and $\la_{\si}=1$ if $\si=-1$.
\end{theorem}

\begin{example}
The following evaluations can not be obtained from Eqs. (\ref{Rpq}) and (\ref{R1q}), but can be deduced by Theorem \ref{Th.Rpq.CMZV}:
\begin{align*}
&R_{2,2}=32\Lii_4(\tf{1}{2})+28\ln(2)\ze(3)+\tf{4}{3}\ln(2)^4
    -\tf{4}{3}\pi^2\ln(2)^2-\tf{121}{360}\pi^4\,,\\
&R_{2,4}=128\ze(\bar{5},1)+\ze(3)^2-\tf{1}{210}\pi^6\,,\\
&R_{3,3}=-192\ze(\bar{5},1)+\tf{109}{2}\ze(3)^2-\tf{8}{105}\pi^6\,,\\
&R_{4,2}=128\ze(\bar{5},1)-55\ze(3)^2+\tf{101}{1260}\pi^6\,.
\end{align*}
and
\begin{align*}
&R_{\bar{1},3}=7\ln(2)\ze(3)+8G^2-\tf{1}{8}\pi^4\,,\\
&R_{\bar{2},2}=-16\Lii_4(\tf{1}{2})-14\ln(2)\ze(3)-8G^2-\tf{2}{3}\ln(2)^4
    +\tf{2}{3}\pi^2\ln(2)^2+\tf{181}{720}\pi^4\,,\\
&R_{1,\bar{2}}=8\imm(\Li_3(\tf{1}{2}+\tf{\ui}{2}))
    -4\ln(2)G-\tf{1}{4}\pi\ln(2)^2-\tf{1}{16}\pi^3\,,\\
&R_{2,\bar{1}}=8\imm(\Li_3(\tf{1}{2}+\tf{\ui}{2}))
    +4\ln(2)G-\tf{1}{4}\pi\ln(2)^2-\tf{11}{48}\pi^3\,,\\
&R_{\bar{1},\bar{2}}=8\imm(\Li_3(\tf{1}{2}+\tf{\ui}{2}))
    +8\ln(2)G-\tf{1}{4}\pi\ln(2)^2-\tf{5}{16}\pi^3\,,\\
&R_{\bar{2},\bar{1}}=-16\imm(\Li_3(\tf{1}{2}+\tf{\ui}{2}))
    -8\ln(2)G+\tf{1}{2}\pi\ln(2)^2+\tf{5}{12}\pi^3\,.
\end{align*}
\end{example}

Finally, we give a transformation theorem of the linear $R$-sums.

\begin{theorem}\label{Th.Rpq.MTV}
For $(p,\si_1)\neq (1,1)$ and $(q,\si_2)\neq (1,1)$, the following relation between the linear $R$-sums and the Kaneko-Tsumura double $T$-values holds:
\begin{equation}
R_{p,q}^{\si_1,\si_2}=\si_1\si_2\ze(p;\si_1)\t(q;\si_2)
    -2^{p+q-2}\si_2T(p,q;\si_1,\si_2)\,.
\end{equation}
\end{theorem}

\begin{proof}
By definitions, we have
\begin{align*}
R_{p,q}^{\si_1,\si_2}
&=\sum_{n=1}^{\infty}\si_2^{n-1}\frac{H_{n-1}^{(p)}(\si_1)}{(n-1/2)^q}
    =\sum_{n=1}^{\infty}\frac{\si_2^{n-1}}{(n-1/2)^q}\sum_{k=1}^{n-1}\frac{\si_1^{k-1}}{k^p}\\
&=\si_1\si_2\sum_{k=1}^{\infty}\frac{\si_1^k}{k^p}
        \sum_{n=1}^{\infty}\frac{\si_2^n}{(n-1/2)^q}
        -\sum_{k\geq n\geq 1}\frac{\si_1^{k-1}\si_2^{n-1}}{k^p(n-1/2)^q}\,,
\end{align*}
from which, the relation given in the theorem follows.
\end{proof}


\section{Quadratic sums}\label{Sec.qua}

In this section, we establish general identities on sums related to three real sequences $A=\{a_k\}_{k\in\mathbb{Z}}$, $B=\{b_k\}_{k\in\mathbb{Z}}$ and $C=\{c_k\}_{k\in\mathbb{Z}}$, which satisfy the same condition as the sequences $A,B$ in Section \ref{Sec.lin}. From the general identities, some (alternating) quadratic Euler $T$-sums and Euler $\S$-sums can be evaluated.


\subsection{The third general identity and quadratic Euler $T$-sums}

Now, we give the third general identity.

\begin{theorem}\label{Th.qua.sum1}
For integers $m,p\geq1$ and $q\geq2$, the following identity on sums related to the sequences $A,B,C$ holds:
\begin{align*}
&(-1)^{m+p+q}\sum_{n=1}^\infty\frac{a_{n-1}\bar{N}\b_n(m)\bar{N}\c_n(p)}{(n-1/2)^q}
    +\sum_{n=1}^\infty\frac{a_nN\b_n(m)N\c_n(p)}{(n-1/2)^q}\\
&\quad-(-1)^{m+p}\sum_{k=0}^{m+p-1}\binom{m+p+q-k-2}{q-1}
    \sum_{n=1}^\infty\frac{b_nc_nS\a_{n+1}(k+1)}{n^{m+p+q-k-1}}\\
&\quad-(-1)^m\sum_{j=1}^m\sum_{k=0}^{m-j}\binom{p+j-2}{p-1}\binom{m+q-j-k-1}{q-1}
    \sum_{n=1}^\infty\frac{b_nM\c_n(p+j-1)S\a_{n+1}(k+1)}{n^{m+q-j-k}}\\
&\quad-(-1)^p\sum_{j=1}^p\sum_{k=0}^{p-j}\binom{m+j-2}{m-1}\binom{p+q-j-k-1}{q-1}
    \sum_{n=1}^\infty\frac{c_nM\b_n(m+j-1)S\a_{n+1}(k+1)}{n^{p+q-j-k}}\\
&\quad=-\mathcal{G}_0\,,
\end{align*}
where
\begin{align*}
\mathcal{G}_0
&=-b_0c_0\{(-1)^{m+p+q}\hat{t}\a(m+p+q)+\t\a(m+p+q)\}\\
&\quad+b_0(-1)^p\sum_{j=1}^{m+q}\binom{p+j-2}{p-1}D\c(p+j-1)\check{t}\a(m+q-j+1)\\
&\quad+c_0(-1)^m\sum_{j=1}^{p+q}\binom{m+j-2}{m-1}D\b(m+j-1)\check{t}\a(p+q-j+1)\\
&\quad
\begin{aligned}
+(-1)^{m+p}\sum_{\substack{j_1+j_2+j_3=q+2\\j_1,j_2,j_3\geq 1}}
    &\binom{m+j_1-2}{m-1}\binom{p+j_2-2}{p-1}\\
    &\times D\b(m+j_1-1)D\c(p+j_2-1)\check{t}\a(j_3)\,.
\end{aligned}
\end{align*}
\end{theorem}

\begin{proof}
Consider the kernel function
\begin{equation}\label{kerfun.xi2}
\xi_2(s)=\pi\cot(\pi s;A)\frac{\vPsi^{(m-1)}(\tf{1}{2}-s;B)\vPsi^{(p-1)}(\tf{1}{2}-s;C)}
    {(m-1)!(p-1)!}
\end{equation}
and the base function $r_1(s)=(s-1/2)^{-q}$. It is obvious that the function $\mathcal{G}(s)=\xi_2(s)r_1(s)$ has simple poles at $s=-n$ for $n\geq 0$, with residues
\[
{\rm Res}(\mathcal{G}(s),-n)
    =(-1)^{m+p+q}\frac{a_n\bar{N}\b_{n+1}(m)\bar{N}\c_{n+1}(p)}{(n+1/2)^q}\,,
\]
and simple poles at $s=n$ for $n\geq 1$, with residues
\[
{\rm Res}(\mathcal{G}(s),n)
    =\frac{a_nN\b_n(m)N\c_n(p)}{(n-1/2)^q}\,,
\]
where Lemmas \ref{Lem.Psi.nZ} and \ref{Lem.cot.nZ} are used. Next, $\mathcal{G}(s)$ has poles of order $p+m$ at $s=n-1/2$ for $n\geq 2$. By Lemmas \ref{Lem.Psi.hn} and \ref{Lem.cot.hn}, we find that the residues are
\begin{align*}
&{\rm Res}(\mathcal{G}(s),n-\tf{1}{2})
    =-(-1)^{m+p}\sum_{k=0}^{m+p-1}\binom{m+p+q-k-2}{q-1}
        \frac{b_{n-1}c_{n-1}S\a_{n}(k+1)}{(n-1)^{m+p+q-k-1}}\\
&\quad-(-1)^m\sum_{j=1}^m\sum_{k=0}^{m-j}\binom{p+j-2}{p-1}\binom{m+q-j-k-1}{q-1}
    \frac{b_{n-1}M\c_{n-1}(p+j-1)S\a_{n}(k+1)}{(n-1)^{m+q-j-k}}\\
&\quad-(-1)^p\sum_{j=1}^p\sum_{k=0}^{p-j}\binom{m+j-2}{m-1}\binom{p+q-j-k-1}{q-1}
    \frac{c_{n-1}M\b_{n-1}(m+j-1)S\a_{n}(k+1)}{(n-1)^{p+q-j-k}}\,.
\end{align*}
Moreover, $\mathcal{G}(s)$ has a pole of order $m+p+q$ at $s=1/2$. Using Eqs. (\ref{Psi.s.0.5}) and (\ref{cot.dm.lim}), the residue ${\rm Res}(\mathcal{G}(s),1/2)$ is found to be $\mathcal{G}_0$ given in the theorem. Hence, combining these four residue results, we obtain the desired result.
\end{proof}

Denote
\begin{align*}
&\begin{aligned}
\mathcal{G}_1=-\si_1\si_2\si_3
    \bibb{
        \begin{array}{c}
            (-1)^m\de_{\si_2\si_3}^{p+q}\t(m;\si_1)T_{p,q}^{\si_2,\si_3}
                +(-1)^p\de_{\si_1\si_3}^{m+q}\t(p;\si_2)T_{m,q}^{\si_1,\si_3}\\
            +T_{m,p+q}^{\si_1,\si_2\si_3}+T_{p,m+q}^{\si_2,\si_1\si_3}
        \end{array}
    }\,,
\end{aligned}\\
&\begin{aligned}
\mathcal{G}_2&=-\si_1(-1)^m\t(m;\si_1)\t(p+q;\si_2\si_3)
        -\si_2(-1)^p\t(p;\si_2)\t(m+q;\si_1\si_3)\\
    &\quad-\si_1\si_2(-1)^{m+p}\de_{\si_3}^{q-1}\t(m;\si_1)\t(p;\si_2)\t(q;\si_3)
        +\si_1\si_2\si_3(-1)^{m+p+q}\t(m+p+q;\si_1\si_2\si_3)\,,
\end{aligned}\\
&\begin{aligned}
\mathcal{G}_3=\si_1\si_2\si_3(-1)^{m+p}
    \sum_{k=0}^{m+p-1}&\binom{m+p+q-k-2}{q-1}\de_{\si_1\si_2\si_3}^k\t(k+1;\si_1\si_2\si_3)\\
    &\times\ze(m+p+q-k-1;\si_3)\,,
\end{aligned}\\
&\begin{aligned}
\mathcal{G}_4&=(-1)^m\sum_{j=1}^m\sum_{k=0}^{m-j}
    \binom{p+j-2}{p-1}\binom{m+q-j-k-1}{q-1}
        \de_{\si_1\si_2\si_3}^k\t(k+1;\si_1\si_2\si_3)\\
    &\qquad\qquad\qquad\qquad\times\bibb{
        \begin{array}{c}
            \si_1S_{p+j-1,m+q-j-k}^{\si_2,\si_3}\\
            -\si_1\si_2\si_3(-1)^{p+j}\ze(p+j-1;\si_2)\ze(m+q-j-k;\si_3)
        \end{array}
        }\\
&\quad+(-1)^p\sum_{j=1}^p\sum_{k=0}^{p-j}
    \binom{m+j-2}{m-1}\binom{p+q-j-k-1}{q-1}
        \de_{\si_1\si_2\si_3}^k\t(k+1;\si_1\si_2\si_3)\\
    &\qquad\qquad\qquad\qquad\times\bibb{
        \begin{array}{c}
            \si_2S_{m+j-1,p+q-j-k}^{\si_1,\si_3}\\
            -\si_1\si_2\si_3(-1)^{m+j}\ze(m+j-1;\si_1)\ze(p+q-j-k;\si_3)
        \end{array}
        }\,,
\end{aligned}\\
&\begin{aligned}
\mathcal{G}_5&=(-1)^m\sum_{j=1}^{p+q}\binom{m+j-2}{m-1}
    \de_{\si_1\si_2\si_3}^{p+q-j}\t(p+q-j+1;\si_1\si_2\si_3)\ze(m+j-1;\si_1)\\
&\quad+(-1)^p\sum_{j=1}^{m+q}\binom{p+j-2}{p-1}
    \de_{\si_1\si_2\si_3}^{m+q-j}\t(m+q-j+1;\si_1\si_2\si_3)\ze(p+j-1;\si_1)\,,
\end{aligned}\\
&\begin{aligned}
\mathcal{G}_6=(-1)^{m+p}\sum_{\substack{j_1+j_2+j_3=q+2\\j_1,j_2,j_3\geq 1}}
    &\binom{m+j_1-2}{m-1}\binom{p+j_2-1}{p-1}\\
    &\times\ze(m+j_1-1)\ze(p+j_2-1;\si_2)
        \de_{\si_1\si_2\si_3}^{j_3-1}\t(j_3;\si_1\si_2\si_3)\,,
\end{aligned}
\end{align*}
where $S_{p,q}^{\si_1,\si_2}$ are the classical linear Euler sums, written in an unified way similar to Eqs. (\ref{Tsum.Unify}) and (\ref{Stsum.Unify}). Thus, the next result holds.

\begin{theorem}\label{Th.qua.Tsum}
The quadratic Euler $T$-sums $T_{m,p,q}^{\si_1,\si_2,\si_3}$ satisfy
\[
\si_1\si_2\si_3\de_{\si_1\si_2\si_3}^{m+p+q-1}T_{m,p,q}^{\si_1,\si_2,\si_3}
    =\mathcal{G}_1+\cdots+\mathcal{G}_6\,,
\]
where $\mathcal{G}_i$, for $i=1,\ldots,6$, are defined as above.
\end{theorem}

\begin{proof}
For integers $m,p\geq1$ and $q\geq2$, letting $(A,B,C)$ be $(A_1,A_1,A_1)$ and $(A_2,A_1,A_1)$ in Theorem \ref{Th.qua.sum1} yields the results of $T_{mp,q}$ and $T_{mp,\bar{q}}$ directly, which coincide with the identity given in Theorem \ref{Th.qua.Tsum}.

Next, setting $(A,B,C)$ by $(A_1,A_2,A_2)$ and $(A_2,A_2,A_2)$ gives the expressions of $T_{\bar{m}\bar{p},q}$ and $T_{\bar{m}\bar{p},\bar{q}}$, respectively. Note that in the transformations, by considering the four cases according to whether or not $m,p$ equal 1, and using the expressions of $T_{\bar{p},\bar{q}}$ and $T_{\bar{p},q}$, we can eliminate the combinations of the terms on the Kronecker delta, which arise when replacing $A,B,C$ by $A_2$. Similarly, using the results of $T_{p,q}$ and $T_{p,\bar{q}}$ in Theorem \ref{Th.lin.Tsum} to eliminate the terms on the Kronecker delta, we establish the expressions of the quadratic $T$-sums of the other forms.

It can be verified that when $q=1$, the sums $T_{m,p,q}^{\si_1,\si_2,\si_3}$ with $\si_3=-1$ also satisfy the identity in this theorem.
\end{proof}

Theorem \ref{Th.qua.Tsum} leads us to the following parity theorem of the quadratic Euler $T$-sums.

\begin{corollary}\label{Coro.Tmpq.PT}
The quadratic Euler $T$-sums $T_{m,p,q}^{\si_1,\si_2,\si_3}$ reduce to combinations of zeta values, the Dirichlet beta values, the classical linear Euler sums and the linear Euler $T$-sums, whenever the weight $w=m+p+q$ and the quantity $\frac{1-\si_1\si_2\si_3}{2}$ are of the same parity.
\end{corollary}

In particular, based on the evaluations of the linear Euler $T$-sums listed in Sections \ref{Sec.lin.Tsum} and \ref{Sec.linES.CMZV}, some quadratic sums are reducible directly to known constants.

\begin{example}
We present the following two sums of odd weight:
\begin{align*}
&T_{11,\bar{1}}=16\imm(\Li_3(\tf{1}{2}+\tf{\ui}{2}))-\tf{1}{3}\pi^3\,,\\
&T_{\bar{1}\bar{1},\bar{1}}=8\imm(\Li_3(\tf{1}{2}+\tf{\ui}{2}))
    -\tf{1}{4}\pi\ln(2)^2-\tf{11}{48}\pi^3\,,
\end{align*}
and the following eleven ones of even weight:
\begin{align*}
&T_{11,2}=16\Lii_4(\tf{1}{2})+\tf{2}{3}\ln(2)^4+\tf{4}{3}\pi^2\ln(2)^2-\tf{23}{360}\pi^4\,,\\
&T_{11,4}=32\ze(\bar{5},1)-62\ln(2)\ze(5)-\tf{17}{2}\ze(3)^2-2\pi^2\ln(2)\ze(3)+\tf{2}{3}\pi^4\ln(2)^2
    +\tf{37}{840}\pi^6\,,\\
&T_{12,3}=16\ze(\bar{5},1)-31\ln(2)\ze(5)-\tf{31}{4}\ze(3)^2+3\pi^2\ln(2)\ze(3)
    +\tf{13}{1008}\pi^6\,,\\
&T_{13,2}=16\ze(\bar{5},1)-31\ln(2)\ze(5)-\tf{33}{2}\ze(3)^2+4\pi^2\ln(2)\ze(3)
    +\tf{181}{5040}\pi^6\,,\\
&T_{22,2}=-7\ze(3)^2+\tf{73}{2520}\pi^6\,,
\end{align*}
and
\begin{align*}
&T_{\bar{1}\bar{1},2}=8\Lii_4(\tf{1}{2})+\tf{1}{3}\ln(2)^4
    +\tf{1}{6}\pi^2\ln(2)^2-\tf{31}{1440}\pi^4\,,\\
&T_{1\bar{1},\bar{2}}=12\Lii_4(\tf{1}{2})+\tf{1}{2}\ln(2)^4
    +\tf{1}{2}\pi^2\ln(2)^2-\tf{101}{960}\pi^4\,,\\
&T_{1\bar{2},\bar{1}}=8\Lii_4(\tf{1}{2})-4G^2+2\pi\ln(2)G+\tf{1}{3}\ln(2)^4
    -\tf{13}{12}\pi^2\ln(2)^2-\tf{23}{720}\pi^4\,,\\
&T_{2\bar{1},\bar{1}}=4\Lii_4(\tf{1}{2})+\tf{1}{6}\ln(2)^4
    -\tf{5}{12}\pi^2\ln(2)^2-\tf{151}{2880}\pi^4\,,\\
&T_{1\bar{3},\bar{2}}=16\ze(\bar{5},1)-31\ln(2)\ze(5)-\tf{17}{4}\ze(3)^2+\tf{9}{4}\pi^2\ln(2)\ze(3)
    +\tf{47}{10080}\pi^6\,,\\
&T_{2\bar{3},\bar{1}}=3\pi^2\Lii_4(\tf{1}{2})-\tf{7}{2}\ze(3)^2
    +\tf{1}{8}\pi^2\ln(2)^4-\tf{1}{8}\pi^4\ln(2)^2-\tf{19}{1008}\pi^6\,.
\end{align*}
\end{example}

\begin{example}
It can be found that the sum $T_{11,\bar{2}}$ can not be covered by Theorem \ref{Th.qua.Tsum} and Corollary \ref{Coro.Tmpq.PT}. As a supplement, we give the evaluation of this sum by CMZVs of level four:
\begin{align*}
T_{11,\bar{2}}&=\sum_{n=1}^{\infty}(-1)^{n-1}\frac{h_{n-1}^2}{(n-1/2)^2}
    =\sum_{n>k_1,k_2\geq1}\frac{(-1)^{n-1}}{(n-1/2)^2(k_1-1/2)(k_2-1/2)}\\
&=4(-1)^{\frac{3}{2}}\sum_{m_1>m_2>m_3\geq 1}
    \frac{\om^{m_1}\prod_{j=1}^3(1-\om^{2m_j})}{m_1^2m_2m_3}
    +\sum_{n>k\geq 1}(-1)^{n-1}\frac{h_{n-1}^{(2)}}{(n-1/2)^2}\\
&=4(-1)^{\frac{3}{2}}\sum_{i,j,k\in\{0,1\}}(-1)^{i+j+k}L_4(2,1,1;2i+1,2j,2k)+T_{2,\bar{2}}\\
&=128\imm(\Li_4(\tf{1}{2}+\tf{\ui}{2}))-96\be(4)
    +\tf{7}{8}\pi\ze(3)+\tf{2}{3}\pi\ln(2)^3+\pi^3\ln(2)\,,
\end{align*}
where $\om=\exp(\pi\ui/2)=\ui$.\hfill\qedsymbol
\end{example}


\subsection{Parity theorem of the Hoffman tripe $t$-values}

Next, let us present the parity theorem of the Hoffman tripe $t$-values.

\begin{theorem}
When the weight $w=m+p+q$ and the quantity $\frac{1-\si_1\si_2\si_3}{2}$ have the same parity, the triple $t$-values $t(m,p,q;\si_1,\si_2,\si_3)$ and triple $\t$-values $\t(m,p,q;\si_1,\si_2,\si_3)$, where $(p,\si_2)\neq(1,1)$ or $(q,\si_3)\neq(1,1)$, are reducible to combinations of zeta values, the Dirichlet beta values, double zeta values and double $t$-values.
\end{theorem}

\begin{proof}
Note that by convention, $(m,\si_1)\neq (1,1)$. Using the harmonic shuffle product, we have
\begin{align}
\t(m,p,q;\si_1,\si_2,\si_3)
    &=\t(m;\si_1)\t(p,q;\si_2,\si_3)\nonumber\\
    &\quad-\t(m+p,q;\si_1\si_2,\si_3)-\si_1\si_2\si_3T_{m,q,p}^{\si_1,\si_3,\si_2}
        \,,\label{Tmpq.TMtV.1}
\end{align}
where $(m,\si_1)\neq(1,1)$ and $(p,\si_2)\neq(1,1)$. According to \cite[Theorems 7.1 and 7.2]{FlSa98} (see also \cite[Corollary 5.1]{XuWang19.EFES}), the classical linear Euler sums $S_{p,q}^{\si_1,\si_2}$ are expressible in terms of zeta values and double zeta values:
\[
S_{p,q}^{\si_1,\si_2}=\si_1\si_2\{\ze(q,p;\si_2,\si_1)+\ze(p+q;\si_1\si_2)\}\,.
\]
Moreover, based on Eq. (\ref{Tpq.DMtV}), the linear $T$-sums $T_{p,q}^{\si_1,\si_2}$ can be converted into double $t$-values. Thus, Corollary \ref{Coro.Tmpq.PT} asserts that, when $w$ and $\frac{1-\si_1\si_2\si_3}{2}$ have the same parity, the sums $T_{m,q,p}^{\si_1,\si_3,\si_2}$, hence the triple values $\t(m,p,q;\si_1,\si_2,\si_3)$ by Eq. (\ref{Tmpq.TMtV.1}), are reducible to zeta values, beta values, double zeta values and double $t$-values, for $(m,\si_1)\neq(1,1)$ and $(p,\si_2)\neq(1,1)$.

On the other hand, by transformations, we have
\begin{equation}\label{Tmpq.TMtV.2}
\t(m,p,q;\si_1,\si_2,\si_3)+\t(m,q,p;\si_1,\si_3,\si_2)+\t(m,p+q;\si_1,\si_2\si_3)
    =\si_1\si_2\si_3T_{p,q,m}^{\si_2,\si_3,\si_1}\,.
\end{equation}
When $w$ and $\frac{1-\si_1\si_2\si_3}{2}$ have the same parity, the sums $T_{p,q,m}^{\si_2,\si_3,\si_1}$ are reducible in the same manner. Therefore, if $(m,\si_1)\neq(1,1)$ and $(p,\si_2)=(1,1)$, but $(q,\si_3)\neq(1,1)$, by the previous discussion, the triple $\t$-values $\t(m,q,p;\si_1,\si_3,\si_2)$ are reducible, so do the triple values $\t(m,p,q;\si_1,\si_2,\si_3)$. This completes the proof.
\end{proof}

Combining with Eqs. (\ref{Tmpq.TMtV.1}) and (\ref{Tmpq.TMtV.2}) with the values of tripe $t$-values given by Hoffman \cite[Appendix A]{Hoff19.AOV}, we evaluate all the non-alternating quadratic $T$-sums of weights $5$ and $7$, which can not be covered by Theorem \ref{Th.qua.Tsum}.

\begin{example}
We have
\begin{align*}
T_{11,3}&=64\Lii_5(\tf{1}{2})-\tf{217}{4}\ze(5)-\tf{8}{15}\ln(2)^5
    +\tf{1}{2}\pi^2\ze(3)+\tf{8}{9}\pi^2\ln(2)^3+\tf{23}{90}\pi^4\ln(2)\,,\\
T_{12,2}&=\tf{93}{8}\ze(5)-\tf{1}{2}\pi^2\ze(3)+\tf{1}{12}\pi^4\ln(2)\,,\\
T_{11,5}&=\tf{569}{34}\ze(7)+\tf{512}{17}\ze(\bar{5},1,1)-\tf{448}{17}\ze(\bar{3},3,1)
    -128\ln(2)\ze(\bar{5},1)+124\ln(2)^2\ze(5)\\
&\quad+34\ln(2)\ze(3)^2+\tf{547}{102}\pi^2\ze(5)-\tf{89}{1020}\pi^4\ze(3)
    -\tf{73}{315}\pi^6\ln(2)\,,\\
T_{12,4}&=\tf{24205}{272}\ze(7)+\tf{384}{17}\ze(\bar{5},1,1)-\tf{64}{17}\ze(\bar{3},3,1)
    +14\ln(2)\ze(3)^2-\tf{57}{17}\pi^2\ze(5)\\
&\quad-\tf{2857}{6120}\pi^4\ze(3)-\tf{2}{105}\pi^6\ln(2)\,,\\
T_{13,3}&=\tf{28865}{136}\ze(7)+\tf{1344}{17}\ze(\bar{5},1,1)-\tf{224}{17}\ze(\bar{3},3,1)
    +49\ln(2)\ze(3)^2-\tf{373}{68}\pi^2\ze(5)\\
&\quad-\tf{8087}{6120}\pi^4\ze(3)-\tf{1}{15}\pi^6\ln(2)\,,\\
T_{14,2}&=-\tf{2615}{272}\ze(7)-\tf{384}{17}\ze(\bar{5},1,1)+\tf{64}{17}\ze(\bar{3},3,1)
    -14\ln(2)\ze(3)^2+\tf{7}{68}\pi^2\ze(5)\\
&\quad+\tf{791}{3060}\pi^4\ze(3)+\tf{11}{210}\pi^6\ln(2)\,,\\
T_{22,3}&=-\tf{127}{8}\ze(7)-\tf{5}{2}\pi^2\ze(5)+\tf{5}{12}\pi^4\ze(3)\,,\\
T_{23,2}&=\tf{635}{16}\ze(7)+\tf{5}{4}\pi^2\ze(5)-\tf{1}{6}\pi^4\ze(3)\,.
\end{align*}
Note that by \cite[Corollaries 4.1 and 4.2]{Hoff19.AOV}, all the non-alternating quadratic Euler $T$-sums can be expressed in terms of alternating MZVs.
\end{example}

In addition, based on the transformation formulas (\ref{Tpq.DMtV}), (\ref{Tmpq.TMtV.1}) and (\ref{Tmpq.TMtV.2}) as well as the linear and quadratic Euler $T$-sums listed in this paper, the evaluations of more Hoffman's double and triple $t$-values, which have not been involved in \cite[Appendix A]{Hoff19.AOV}, can be established.

\begin{example}
Here are some selected alternating Hoffman tripe $t$-values:
\begin{align*}
t(2,\bar{1},2)&=\tf{1}{8}\pi\Lii_4(\tf{1}{2})+\tf{1}{192}\pi\ln(2)^4
    -\tf{1}{8}\pi^2\imm(\Li_3(\tf{1}{2}+\tf{\ui}{2}))
    -\tf{1}{768}\pi^3\ln(2)^2+\tf{41}{23040}\pi^5\,,\\
t(\bar{2},\bar{1},\bar{2})&=\tf{7}{16}\ze(3)G-\tf{1}{8}\pi\Lii_4(\tf{1}{2})
    -\tf{1}{192}\pi\ln(2)^4-\tf{1}{16}\pi^2\ln(2)G
    +\tf{1}{192}\pi^3\ln(2)^2+\tf{1}{5760}\pi^5\,,\\
t(3,\bar{1},3)&=\tf{7}{16}\ze(3)\be(4)-\tf{3}{64}\pi\ze(\bar{5},1)
    +\tf{93}{1024}\pi\ln(2)\ze(5)-\tf{75}{4096}\pi\ze(3)^2\\
&\quad-\tf{7}{1024}\pi^3\ln(2)\ze(3)-\tf{391}{2580480}\pi^7\,,\\
t(\bar{3},\bar{1},\bar{3})&=+\tf{3}{64}\pi\ze(\bar{5},1)
    -\tf{93}{1024}\pi\ln(2)\ze(5)-\tf{9}{512}\pi\ze(3)^2
    -\tf{1}{128}\pi^3\Lii_4(\tf{1}{2})-\tf{1}{3072}\pi^3\ln(2)^4\\
&\quad+\tf{1}{3072}\pi^5\ln(2)^2+\tf{69}{573440}\pi^7\,,
\end{align*}
and
\begin{align*}
t(2,\bar{1},\bar{1})&=\tf{1}{4}\Lii_4(\tf{1}{2})+\tf{1}{96}\ln(2)^4
    +\tf{1}{192}\pi^2\ln(2)^2-\tf{91}{46080}\pi^4\,,\\
t(\bar{2},1,\bar{1})&=\tf{1}{4}\Lii_4(\tf{1}{2})-\tf{1}{4}G^2+\tf{1}{96}\ln(2)^4
    +\tf{1}{192}\pi^2\ln(2)^2+\tf{29}{46080}\pi^4\,,\\
t(\bar{2},\bar{1},1)&=\tf{1}{2}\Lii_4(\tf{1}{2})-\tf{1}{4}G^2+\tf{1}{48}\ln(2)^4
    +\tf{5}{192}\pi^2\ln(2)^2-\tf{23}{11520}\pi^4\,,\\
t(\bar{2},2,\bar{2})&=-\tf{1}{2}\be(4)G+\tf{7}{64}\pi\ze(3)G
    -\tf{1}{16}\pi^2\Lii_4(\tf{1}{2})-\tf{7}{128}\pi^2\ln(2)\ze(3)
    -\tf{1}{384}\pi^2\ln(2)^4\\
&\quad+\tf{1}{384}\pi^4\ln(2)^2+\tf{7}{9216}\pi^6\,,\\
t(\bar{2},4,\bar{2})&=-\tf{1}{2}\be(6)G+\tf{31}{512}\pi\ze(5)G-\tf{1}{64}\pi^2\ze(\bar{5},1)
    -\tf{3}{4096}\pi^2\ze(3)^2+\tf{3}{512}\pi^3\ze(3)G+\tf{1}{103680}\pi^8\,.
\end{align*}
Moreover, we present two representative reduction formulas:
\begin{align*}
t(\bar{4},\bar{1},1)&=\tf{1}{4}\ze(\bar{5},1)+\tf{1}{2}t(\bar{2},\bar{4})
    -\tf{31}{64}\ln(2)\ze(5)-\tf{1}{2}\be(4)G-\tf{17}{256}\ze(3)^2
    +\tf{3}{256}\pi^2\ze(3)\ln(2)\\
&\quad+\tf{1}{64}\pi^2\Lii_4(\tf{1}{2})+\tf{1}{1536}\pi^2\ln(2)^4
    +\tf{5}{1536}\pi^4\ln(2)^2+\tf{101}{161280}\pi^6\,,\\
t(5,\bar{1},1)&=-\tf{1}{2}t(\bar{6},1)+\tf{1}{2}t(\bar{2},5)+\tf{31}{64}\ze(5)G
    -\tf{1}{256}\pi\ze(\bar{5},1)+\tf{33}{8192}\pi\ze(3)^2
    -\tf{3}{1024}\ln(2)\pi^3\ze(3)\\
&\quad-\tf{1}{6144}\pi^3\ln(2)^4-\tf{13}{12288}\pi^5\ln(2)^2
    -\tf{1}{256}\pi^3\Lii_4(\tf{1}{2})-\tf{383}{5160960}\pi^7\,.
\end{align*}
The readers may try to find more evaluations and reduction formulas of the Hoffman double and triple $t$-values by the results of the present paper.\hfill\qedsymbol
\end{example}


\subsection{The fourth general identity and quadratic Euler $\S$-sums}\label{Sec.qua.Stsum}

Using the same kernel function as that of Theorem \ref{Th.qua.sum1} but different base function, we obtain the fourth identity.

\begin{theorem}\label{Th.qua.sum2}
For integers $m,p\geq1$ and $q\geq2$, the following identity on sums related to the sequences $A,B,C$ holds:
\begin{align*}
&(-1)^{m+p+q}\sum_{n=1}^\infty\frac{a_n\bar{N}\b_{n+1}(m)\bar{N}\c_{n+1}(p)}{n^q}
    +\sum_{n=1}^\infty\frac{a_nN\b_n(m)N\c_n(p)}{n^q}\\
&-(-1)^{m+p}\sum_{k=0}^{m+p-1}\binom{m+p+q-k-2}{q-1}
    \sum_{n=1}^\infty\frac{b_{n-1}c_{n-1}S\a_{n}(k+1)}{(n-1/2)^{m+p+q-k-1}}\\
&-(-1)^m\sum_{j=1}^m\sum_{k=0}^{m-j}\binom{p+j-2}{p-1}\binom{m+q-j-k-1}{q-1}
    \sum_{n=1}^\infty\frac{b_{n-1}M\c_{n-1}(p+j-1)S\a_{n}(k+1)}{(n-1/2)^{m+q-j-k}}\\
&-(-1)^p\sum_{j=1}^p\sum_{k=0}^{p-j}\binom{m+j-2}{m-1}\binom{p+q-j-k-1}{q-1}
    \sum_{n=1}^\infty\frac{c_{n-1}M\b_{n-1}(m+j-1)S\a_{n}(k+1)}{(n-1/2)^{p+q-j-k}}\\
&=-\mathcal{H}_0\,,
\end{align*}
where
\begin{align*}
\mathcal{H}_0
&=a_0(-1)^{m+p}\sum_{\substack{k_1+k_2=q\\k_1,k_2\geq 0}}
    \binom{m+k_1-1}{m-1}\binom{p+k_2-1}{p-1}\hat{t}\b(m+k_1)\hat{t}\c(p+k_2)\\
&\quad\begin{aligned}
-2(-1)^{m+p}\sum_{j=1}^{[\frac{q}{2}]}
    \sum_{\substack{k_1+k_2=q-2j\\k_1,k_2\geq 0}}
    &\binom{m+k_1-1}{m-1}\binom{p+k_2-1}{p-1}\\
    &\times D\a(2j)\hat{t}\b(m+k_1)\hat{t}\c(p+k_2)\,.
\end{aligned}
\end{align*}
\end{theorem}

\begin{proof}
The theorem results from applying the kernel function $\xi_2(s)$ in (\ref{kerfun.xi2}) to the base function $r_2(s)=s^{-q}$, and performing the residue computation. If we define $\mathcal{H}(s)=\xi_2(s)r_2(s)$, then the quantity $\mathcal{H}_0$ equals ${\rm Res}(\mathcal{H}(s),0)$, and the other quantities of the general identity represent combined contributions of the poles at $s=\pm n$ and $s=n-1/2$.
\end{proof}

By using the linear sums $R_{p,q}^{\si_1,\si_2}$ discussed in Section \ref{Sec.lin.Rsum}, we further denote
\begin{align*}
&\begin{aligned}
\mathcal{H}_1&=-\si_1\si_2\si_3(-1)^m\de_{\si_2}^{p+q}\t(m;\si_1)\S_{p,q}^{\si_2,\si_3}
    -\si_1\si_2\si_3(-1)^p\de_{\si_1}^{m+q}\t(p;\si_2)\S_{m,q}^{\si_1,\si_3}\\
&\quad-\si_1\si_2(-1)^{m+p}(1+(-1)^q)\t(m;\si_1)\t(p;\si_2)\ze(q;\si_3)\,,
\end{aligned}\\
&\begin{aligned}
\mathcal{H}_2=\si_1\si_2(-1)^{m+p}\sum_{k=0}^{m+p-1}
    &\binom{m+p+q-k-2}{q-1}\de_{\si_1\si_2\si_3}^k\\
    &\times\t(k+1;\si_1\si_2\si_3)\t(m+p+q-k-1;\si_3)\,,
\end{aligned}\\
&\begin{aligned}
\mathcal{H}_3&=\si_1\si_3(-1)^m\sum_{j=1}^m\sum_{k=0}^{m-j}
    \binom{p+j-2}{p-1}\binom{m+q-j-k-1}{q-1}\de_{\si_1\si_2\si_3}^k\t(k+1;\si_1\si_2\si_3)\\
&\qquad\qquad\qquad\qquad\qquad\times\bibb{
    \begin{array}{c}
        R_{p+j-1,m+q-j-k}^{\si_2,\si_3}\\
        -\si_2\si_3(-1)^{p+j}\ze(p+j-1;\si_2)\t(m+q-j-k;\si_3)
    \end{array}
}\\
&\quad+\si_2\si_3(-1)^p\sum_{j=1}^p\sum_{k=0}^{p-j}
    \binom{m+j-2}{m-1}\binom{p+q-j-k-1}{q-1}\de_{\si_1\si_2\si_3}^k\t(k+1;\si_1\si_2\si_3)\\
&\qquad\qquad\qquad\qquad\qquad\times\bibb{
    \begin{array}{c}
        R_{m+j-1,p+q-j-k}^{\si_1,\si_3}\\
        -\si_1\si_3(-1)^{m+j}\ze(m+j-1;\si_1)\t(p+q-j-k;\si_3)
    \end{array}
}\,,
\end{aligned}\\
&\begin{aligned}
\mathcal{H}_4=-\si_1\si_2(-1)^{m+p}\sum_{\substack{k_1+k_2=q\\k_1,k_2\geq 0}}
    \binom{m+k_1-1}{m-1}\binom{p+k_2-1}{p-1}\t(m+k_1;\si_1)\t(p+k_2;\si_2)\,,
\end{aligned}\\
&\begin{aligned}
\mathcal{H}_5=2\si_1\si_2(-1)^{m+p}\sum_{j=1}^{[\frac{q}{2}]}
    \sum_{\substack{k_1+k_2=q-2j\\k_1,k_2\geq 0}}
&\binom{m+k_1-1}{m-1}\binom{p+k_2-1}{p-1}\\
&\times\ze(2j;\si_1\si_2\si_3)\t(m+k_1;\si_1)\t(p+k_2;\si_2)\,.
\end{aligned}
\end{align*}
Then the following result on the quadratic Euler $\S$-sums holds.

\begin{theorem}\label{Th.qua.Stsum}
The quadratic Euler $\S$-sums $\S_{m,p,q}^{\si_1,\si_2,\si_3}$ satisfy
\[
\si_1\si_2\si_3\de_{\si_1\si_2}^{m+p+q-1}\S_{m,p,q}^{\si_1,\si_2,\si_3}
    =\mathcal{H}_1+\cdots+\mathcal{H}_5\,,
\]
where $\mathcal{H}_i$, for $i=1,\ldots,5$, are defined as above.
\end{theorem}

\begin{proof}
For $m,p\geq1$ and $q\geq2$, setting $(A,B,C)$ by $(A_1,A_1,A_1)$ and $(A_2,A_1,A_1)$ in Theorem \ref{Th.qua.sum2} gives immediately the results of $\S_{mp,q}$ and $\S_{mp,\bar{q}}$, which can be covered by this theorem.

Setting $(A,B,C)$ by $(A_1,A_2,A_2)$ and $(A_2,A_2,A_2)$ yields the expressions of $\S_{\bar{m}\bar{p},q}$ and $\S_{\bar{m}\bar{p},\bar{q}}$. Note that when eliminating the terms involving the Kronecker delta, we should use the results of $\S_{\bar{p},\bar{q}}$ and $\S_{\bar{p},q}$ respectively, and discuss the details according to the values of $m$ and $p$. In particular, if $(m,p)=(1,1)$, the simplification process also relies on the parity of $q$. Similarly, using the results of $\S_{p,\bar{q}}$ and $\S_{p,\bar{q}}$ to eliminate the terms on the Kronecker delta, we obtain the expressions of the other quadratic Euler $\S$-sums.

Finally, it can be verified that the conclusion also holds when $q=1$ and $\si_3=-1$.
\end{proof}

From Theorem \ref{Th.qua.Stsum}, we can obtain immediately the following parity theorem of the quadratic Euler $\S$-sums.

\begin{corollary}\label{Coro.Stmpq.PT}
The quadratic Euler $\S$-sums $\S_{m,p,q}^{\si_1,\si_2,\si_3}$ reduce to combinations of zeta values, the Dirichlet beta values, the linear Euler $R$-sums and the linear Euler $\S$-sums, whenever the weight $w=m+p+q$ and the quantity $\frac{1-\si_1\si_2}{2}$ are of the same parity.
\end{corollary}

Moreover, in the specific case of $m=p=1$, the next result holds.

\begin{corollary}\label{Coro.St11q}
The quadratic Euler $\S$-sums $\S_{1,1,q}^{\si_1,\si_2,\si_3}$ are expressible in terms of $\pi$, $G$, zeta values $\ze(3),\ze(5),\ldots$, and the Dirichlet beta values $\be(4),\be(6),\ldots$, whenever the parameter $q$ and the quantity $\frac{1-\si_1\si_2}{2}$ are of the same parity, and $\si_1\si_2\si_3=1$.
\end{corollary}

\begin{proof}
Under these two conditions, we obtain from Theorem \ref{Th.qua.Stsum} the expression of $\S_{11,q}$, where $q$ is even:
\begin{align*}
\S_{11,q}
&=2\t(1)\S_{1,q}-\t(1)^2\ze(q)+\t(2)\t(q)
    -\frac{1}{2}\sum_{\substack{k_1+k_2=q\\k_1,k_2\geq0}}\t(k_1+1)\t(k_2+1)\\
&\quad+\sum_{j=1}^{q/2}\sum_{\substack{k_1+k_2=q-2j\\k_1,k_2\geq0}}\ze(2j)\t(k_1+1)\t(k_2+1)\\
&=\t(2)\t(q)-\frac{1}{2}\sum_{\substack{k_1+k_2=q\\k_1,k_2\geq1}}\t(k_1+1)\t(k_2+1)
    +\sum_{j=1}^{q/2-1}\sum_{\substack{k_1+k_2=q-2j\\k_1,k_2\geq1}}\ze(2j)\t(k_1+1)\t(k_2+1)\,.
\end{align*}
Note that in the last step, the expression of $\S_{1,q}$ is used to eliminate the terms on $\t(1)$. Thus, the final expression of $\S_{11,q}$, where $q$ is even, does not involve the terms on $\t(1)$ and $\ze(\bar{1})$, and coincides with the statement of the corollary. The similar discussion also holds for $\S_{1\bar{1},\bar{q}}$ where $q$ is odd, and $\S_{1\bar{1},\bar{q}}$ where $q$ is even.
\end{proof}

Here are some special quadratic Euler $\S$-sums which satisfy the conditions of Corollary \ref{Coro.St11q}.

\begin{example}
We have
\begin{align*}
\S_{1\bar{1},\bar{1}}&=\tf{1}{8}\pi^3\,,\\
\S_{1\bar{1},\bar{3}}&=-14\ze(3)G+\tf{1}{16}\pi^5\,,\\
\S_{1\bar{1},\bar{5}}&=-62\ze(5)G-56\ze(3)\be(4)+\tf{14}{3}\pi^2\ze(3)G+\tf{5}{192}\pi^7\,,\\
\S_{1\bar{1},\bar{7}}&=
    -254\ze(7)G-248\ze(5)\be(4)-224\ze(3)\be(6)\\
&\quad+\tf{62}{3}\pi^2\ze(5)G
    +\tf{56}{3}\pi^2\ze(3)\be(4)+\tf{14}{45}\pi^4\ze(3)G+\tf{61}{5760}\pi^9\,,\\
\S_{1\bar{1},\bar{9}}&=-1022\ze(9)G-1016\ze(7)\be(4)
    -992\ze(5)\be(6)-896\ze(3)\be(8)\\
&\quad+\tf{254}{3}\pi^2\ze(7)G+\tf{248}{3}\pi^2\ze(5)\be(4)
    +\tf{224}{3}\pi^2\ze(3)\be(6)
    +\tf{62}{45}\pi^4\ze(5)G+\tf{56}{45}\pi^4\ze(3)\be(4)\\
&\quad+\tf{4}{135}\pi^6\ze(3)G+\tf{277}{64512}\pi^{11}\,,
\end{align*}
and
\begin{align*}
\S_{11,2}&=\tf{1}{8}\pi^4\,,\\
\S_{11,4}&=-\tf{49}{2}\ze(3)^2+\tf{1}{24}\pi^6\,,\\
\S_{11,6}&=-217\ze(3)\ze(5)+\tf{49}{6}\pi^2\ze(3)^2+\tf{1}{60}\pi^8\,,\\
\S_{11,8}&=-889\ze(3)\ze(7)-\tf{961}{2}\ze(5)^2+\tf{217}{3}\pi^2\ze(3)\ze(5)
    +\tf{49}{90}\pi^4\ze(3)^2+\tf{17}{2520}\pi^{10}\,,\\
\S_{\bar{1}\bar{1},2}&=-8G^2+\tf{1}{8}\pi^4\,,\\
\S_{\bar{1}\bar{1},4}&=-64\be(4)G+\tf{8}{3}\pi^2G^2+\tf{1}{24}\pi^6\,,\\
\S_{\bar{1}\bar{1},6}&=-256\be(6)G-128\be(4)^2
    +\tf{64}{3}\pi^2\be(4)G+\tf{8}{45}\pi^4G^2+\tf{1}{60}\pi^8\,,\\
\S_{\bar{1}\bar{1},8}&=-1024\be(8)G-1024\be(4)\be(6)+\tf{256}{3}\pi^2\be(6)G
    +\tf{128}{3}\pi^2\be(4)^2+\tf{64}{45}\pi^4\be(4)G\\
    &\quad+\tf{16}{945}\pi^6G^2+\tf{17}{2520}\pi^{10}\,,
\end{align*}
where the evaluation of $\S_{11,2}$ has been shown in different ways in \cite[Eq. (13)]{BorBor95.OAI}, \cite[Eq. (3.3b)]{Chu97.HSR} and \cite[Eq. (22)]{DeDo91}.\hfill\qedsymbol
\end{example}

By Theorem \ref{Th.qua.Stsum} as well as the evaluations of corresponding linear $\S$-sums and $R$-sums listed in Sections \ref{Sec.lin.Stsum}--\ref{Sec.lin.Rsum}, more quadratic $\S$-sums can be determined directly:

\begin{example}
For even weights, the following quadratic $\S$-sums are reducible to known constants:
\begin{align*}
\S_{12,3}&=49\ze(3)^2-\tf{1}{16}\pi^6\,,\\
\S_{13,2}&=-16\pi^2\Lii_4(\tf{1}{2})-14\pi^2\ln(2)\ze(3)-\tf{2}{3}\pi^2\ln(2)^4
    +\tf{2}{3}\pi^4\ln(2)^2+\tf{151}{720}\pi^6\,,\\
\S_{22,2}&=-98\ze(3)^2+32\pi^2\Lii_4(\tf{1}{2})+28\pi^2\ln(2)\ze(3)
    +\tf{4}{3}\pi^2\ln(2)^4-\tf{4}{3}\pi^4\ln(2)^2-\tf{61}{360}\pi^6\,,\\
\S_{12,5}&=651\ze(3)\ze(5)-\tf{343}{12}\pi^2\ze(3)^2-\tf{1}{24}\pi^8\,,\\
\S_{13,4}&=-651\ze(3)\ze(5)-64\pi^2\ze(\bar{5},1)+\tf{97}{2}\pi^2\ze(3)^2
    +\tf{41}{2520}\pi^8\,,\\
\S_{14,3}&=434\ze(3)\ze(5)+96\pi^2\ze(\bar{5},1)-\tf{193}{4}\pi^2\ze(3)^2
    +\tf{29}{1680}\pi^8\,,\\
\S_{15,2}&=-64\pi^2\ze(\bar{5},1)+\tf{55}{2}\pi^2\ze(3)^2-\tf{16}{3}\pi^4\Lii_4(\tf{1}{2})
    -\tf{14}{3}\pi^4\ln(2)\ze(3)-\tf{2}{9}\pi^4\ln(2)^4\\
    &\quad+\tf{2}{9}\pi^6\ln(2)^2+\tf{493}{15120}\pi^8\,,\\
\S_{22,4}&=-1736\ze(3)\ze(5)+128\pi^2\ze(\bar{5},1)+\tf{101}{3}\pi^2\ze(3)^2
    +\tf{443}{2520}\pi^8\,,\\
\S_{23,3}&=1302\ze(3)\ze(5)-96\pi^2\ze(\bar{5},1)-\tf{3}{4}\pi^2\ze(3)^2
    -\tf{137}{840}\pi^8\,,\\
\S_{24,2}&=-868\ze(3)\ze(5)+64\pi^2\ze(\bar{5},1)-\tf{27}{2}\pi^2\ze(3)^2
    +\tf{16}{3}\pi^4\Lii_4(\tf{1}{2})+\tf{14}{3}\pi^4\ln(2)\ze(3)\\
    &\quad+\tf{2}{9}\pi^4\ln(2)^4-\tf{2}{9}\pi^6\ln(2)^2+\tf{1307}{15120}\pi^8\,,\\
\S_{33,2}&=-28\pi^2\ze(3)^2+\tf{3}{56}\pi^8\,,
\end{align*}
and
\begin{align*}
\S_{11,\bar{2}}&=8\pi\imm(\Li_3(\tf{1}{2}+\tf{\ui}{2}))
    +4\pi\ln(2)G-\tf{1}{4}\pi^2\ln(2)^2-\tf{3}{16}\pi^4\,,\\
\S_{12,\bar{1}}&=-4\pi\imm(\Li_3(\tf{1}{2}+\tf{\ui}{2}))
    -2\pi\ln(2)G+\tf{1}{8}\pi^2\ln(2)^2+\tf{5}{32}\pi^4\,,\\
\S_{\bar{1}\bar{1},\bar{2}}&=8G^2-8\pi\imm(\Li_3(\tf{1}{2}+\tf{\ui}{2}))
    -4\pi\ln(2)G+\tf{1}{4}\pi^2\ln(2)^2+\tf{3}{16}\pi^4\,,\\
\S_{\bar{1}\bar{2},\bar{1}}&=4\pi\imm(\Li_3(\tf{1}{2}+\tf{\ui}{2}))
    +2\pi\ln(2)G-\tf{1}{8}\pi^2\ln(2)^2-\tf{1}{32}\pi^4\,,\\
\S_{\bar{1}\bar{2},3}&=96\be(4)G+14\pi\ze(3)G-8\pi^2G^2-\tf{1}{16}\pi^6\,,\\
\S_{\bar{1}\bar{3},2}&=-96\be(4)G-8\pi^2\Lii_4(\tf{1}{2})-7\pi^2\ln(2)\ze(3)
    +4\pi^2G^2-\tf{1}{3}\pi^2\ln(2)^4\\
    &\quad+\tf{1}{3}\pi^4\ln(2)^2+\tf{241}{1440}\pi^6\,,\\
\S_{\bar{2}\bar{2},2}&=-56\pi\ze(3)G+16\pi^2\Lii_4(\tf{1}{2})+14\pi^2\ln(2)\ze(3)
    +16\pi^2G^2+\tf{2}{3}\pi^2\ln(2)^4\\
    &\quad-\tf{2}{3}\pi^4\ln(2)^2-\tf{61}{720}\pi^6\,.
\end{align*}
A quick glance through this list yields two elegant symmetric series:
\[
\sum_{n=1}^{\infty}(-1)^{n-1}\frac{h_n^2+\bar{h}_n^2}{n^2}
    =\S_{11,\bar{2}}+\S_{\bar{1}\bar{1},\bar{2}}=8G^2\,,
\]
and
\[
\sum_{n=1}^{\infty}(-1)^{n-1}\frac{h_nh_n^{(2)}+\bar{h}_n\bar{h}_n^{(2)}}{n}
    =\S_{12,\bar{1}}+\S_{\bar{1}\bar{2},\bar{1}}=\frac{1}{8}\pi^4\,.
    \hfill\qedsymbol
\]
\end{example}

\begin{example}
For odd weights, the following quadratic $\S$-sums are expressible in terms of known constants:
\begin{align*}
\S_{1\bar{2},\bar{2}}&=8\pi G^2+4\pi^2\imm(\Li_3(\tf{1}{2}+\tf{\ui}{2}))
    +2\pi^2\ln(2)G-\tf{1}{8}\pi^3\ln(2)^2-\tf{5}{32}\pi^5\,,\\
\S_{1\bar{3},\bar{1}}&=-8\pi^2\imm(\Li_3(\tf{1}{2}+\tf{\ui}{2}))
    -4\pi^2\ln(2)G+\tf{1}{4}\pi^3\ln(2)^2+\tf{1}{4}\pi^5\,,\\
\S_{2\bar{1},\bar{2}}&=28\ze(3)G-4\pi^2\imm(\Li_3(\tf{1}{2}+\tf{\ui}{2}))
    -2\pi^2\ln(2)G+\tf{1}{8}\pi^3\ln(2)^2+\tf{1}{32}\pi^5\,,\\
\S_{2\bar{2},\bar{1}}&=-8\pi G^2+12\pi^2\imm(\Li_3(\tf{1}{2}+\tf{\ui}{2}))
    +6\pi^2\ln(2)G-\tf{3}{8}\pi^3\ln(2)^2-\tf{7}{32}\pi^5\,,\\
\S_{3\bar{1},\bar{1}}&=-4\pi^2\imm(\Li_3(\tf{1}{2}+\tf{\ui}{2}))
    -2\pi^2\ln(2)G+\tf{1}{8}\pi^3\ln(2)^2+\tf{5}{32}\pi^5\,,
\end{align*}
and
\begin{align*}
&\S_{1\bar{1},3}=14\ze(3)G-4\pi G^2\,,\\
&\S_{1\bar{2},2}=-8\pi\Lii_4(\tf{1}{2})-7\pi\ln(2)\ze(3)+4\pi G^2-\tf{1}{3}\pi\ln(2)^4
    +\tf{1}{3}\pi^3\ln(2)^2+\tf{151}{1440}\pi^5\,,\\
&\S_{2\bar{1},2}=-28\ze(3)G+8\pi\Lii_4(\tf{1}{2})+7\pi\ln(2)\ze(3)
    +4\pi G^2+\tf{1}{3}\pi\ln(2)^4-\tf{1}{3}\pi^3\ln(2)^2+\tf{29}{1440}\pi^5\,.
\end{align*}
Finally, we evaluate the Euler $\S$-sum $\S_{11,\bar{1}}$ by CMZVs:
\begin{align*}
\S_{11,\bar{1}}&=\sum_{n=1}^{\infty}(-1)^{n-1}\frac{h_n^2}{n}
    =\sum_{n\geq k_1,k_2\geq 1}\frac{(-1)^{n-1}}{n(k_1-1/2)(k_2-1/2)}\\
&=-2\sum_{m_1>m_2>m_3\geq1}\frac{\om^{m_1}(1+\om^{2m_1})(1-\om^{2m_2})(1+\om^{2m_3})}
    {m_1m_2m_3}+\sum_{n=1}^{\infty}(-1)^{n-1}\frac{h_n^{(2)}}{n}\\
&=-2\sum_{i,j,k\in\{0,1\}}(-1)^{j+k}L_4(1,1,1;2i+1,2j,2k)+\S_{2,\bar{1}}
    =\frac{7}{4}\ze(3)\,,
\end{align*}
which can be found in Chu \cite[Eq. (4.5a)]{Chu97.HSR} and De Doelder \cite[Eq. (21)]{DeDo91}, but can not be covered by Theorem \ref{Th.qua.Stsum}.
\hfill\qedsymbol
\end{example}

\begin{example}
Using the results of linear and quadratic Euler $T$-sums and $\S$-sums, we can give the evaluations of some infinite series involving odd harmonic numbers. For example, we have
\begin{align*}
&\sum_{n=1}^{\infty}\frac{(O_{n-1}+O_n)^2}{(2n-1)^2}
    =\frac{1}{4}T_{1^2,2}+\frac{1}{4}T_{1,3}+\frac{1}{16}\t(4)
    =\frac{1}{2}\pi^2\ln^2(2)+\frac{1}{96}\pi^4\,,\\
&\sum_{n=1}^{\infty}(-1)^{n-1}
    \bibb{\frac{O_{n-1}^{(2)}+O_n^{(2)}}{(2n-1)^2}+2\frac{O_{n-1}+O_n}{(2n-1)^3}}
=\frac{1}{8}T_{2,\bar{2}}+\frac{1}{4}T_{1,\bar{3}}-\frac{3}{16}\t(\bar{4})
=\frac{1}{8}\pi^3\ln(2)\,,
\end{align*}
which were established by Chu \cite[Eqs. (3.5a) and (5.1)]{Chu97.HSR} through the Dougalll-Dixon theorem and one of Dougall's bilateral summation theorems. Moreover, based on our previous result \cite[Corollary 4.3]{XuWang19.TVE}, the cubic sum $\S_{1^3,2}$ equals
\[
\S_{1^3,2}=\tf{7}{2}\pi^2\ze(3)\,.
\]
Thus, the following series due to Zheng \cite[Eq. (3.6d)]{ZhengDY07} holds:
\begin{equation}\label{Zheng.3.6d}
\sum_{n=1}^{\infty}\frac{2O_n^3+O_n^{(3)}}{n^2}
    =\frac{1}{4}\S_{1^3,2}+\frac{1}{8}\S_{3,2}
    =\frac{93}{8}\ze(5)\,.
\end{equation}
By transformations, Chu's results \cite[Eqs. (3.4), (3.5b), (3.5c), (4.5a), (4.5c) and (4.6a)-(4.6c)]{Chu97.HSR} and Zheng's results \cite[Eqs. (3.8b), (3.8c), (3.10a), (4.4a) and (4.4b)]{ZhengDY07} can also be rederived immediately.
\hfill\qedsymbol
\end{example}

\begin{example}
Combining the results in this paper and in the literature, we can determine the evaluations of more series. A similar series to (\ref{Zheng.3.6d}) is
\[
\sum_{n=1}^{\infty}(-1)^{n-1}\frac{2O_n^3+O_n^{(3)}}{n}
    =\frac{1}{4}\S_{1^3,\bar{1}}+\frac{1}{8}\S_{3,\bar{1}}=\frac{1}{64}\pi^4
\]
(see \cite[Eq. (4.3d)]{ZhengDY07}), from which, we obtain the triple $\S$-sum $\S_{1^3,\bar{1}}$:
\[
\S_{1^3,\bar{1}}=4G^2\,,
\]
and further verify Zheng's two complicated series \cite[Eqs. (4.3e) and (4.3f)]{ZhengDY07}. Similarly, from the series given in \cite[Eq. (4.5a)]{ZhengDY07}:
\begin{align*}
\sum_{n=1}^{\infty}(-1)^{n-1}\frac{2O_{n-1}^3+2O_n^3+O_{n-1}^{(3)}+O_n^{(3)}}{2n-1}
&=\frac{1}{4}T_{1^3,\bar{1}}+\frac{3}{8}T_{1^2,\bar{2}}+\frac{3}{8}T_{1,\bar{3}}
    +\frac{1}{8}T_{3,\bar{1}}-\frac{3}{16}\t(\bar{4})\\
&=\frac{3}{16}\pi\ze(3)+\frac{1}{8}\pi\ln(2)^3+\frac{1}{32}\pi^3\ln(2)\,,
\end{align*}
we obtain the triple $T$-sum $T_{1^3,\bar{1}}$:
\[
T_{1^3,\bar{1}}=-192\imm(\Li_4(\tf{1}{2}+\tf{\ui}{2}))+148\be(4)
    -\tf{9}{16}\pi\ze(3)-\tf{1}{2}\pi\ln(2)^3-2\pi^3\ln(2)\,,
\]
and further verify the series \cite[Eqs. (4.5b)]{ZhengDY07}.
\hfill\qedsymbol
\end{example}


\subsection{Parity theorem of the Kaneko-Tsumura triple $T$-values}

Now, let us give the parity theorem of the Kaneko-Tsumura triple $T$-values.

\begin{theorem}\label{Th.PT.MTV}
When the weight $w=m+p+q$ and the quantity $\frac{1-\si_1\si_3}{2}$ have the same parity, the triple $T$-values $T(m,p,q;\si_1,\si_2,\si_3)$, where $(p,\si_2)\neq (1,1)$, are reducible to combinations of zeta values, the Dirichlet beta values, linear Euler $R$-sums and double $T$-values.
\end{theorem}

\begin{proof}
Note that by convention, $(m,\si_1)\neq (1,1)$. According to Eq. (\ref{Stpq.DMTV}), we have
\[
\t(m;\si_1)\S_{q,p}^{\si_3,\si_2}
    =2^{m+p+q}\si_3\bibb{\sum_{n_1\geq n_2>n_3\geq1}+\sum_{n_2>n_1,n_3\geq 1}}
        \frac{\si_1^{n_1}\si_2^{n_2}\si_3^{n_3}}{(2n_1-1)^m(2n_2-2)^p(2n_3-1)^q}\,,
\]
which further gives
\begin{equation}\label{Stmpq.TMTV}
T(m,p,q;\si_1,\si_2,\si_3)
    =\frac{\si_1\si_3}{2^{m+p+q-3}}\t(m;\si_1)\S_{q,p}^{\si_3,\si_2}
    -\frac{\si_3}{2^{m+p+q-3}}\S_{m,q,p}^{\si_1,\si_3,\si_2}\,,
\end{equation}
for $(m,\si_1)\neq(1,1)$ and $(p,\si_2)\neq (1,1)$. Thus, by Corollary \ref{Coro.Stmpq.PT} and Eq. (\ref{Stmpq.TMTV}), under the conditions of this theorem, the triple $T$-values $T(m,p,q;\si_1,\si_2,\si_3)$ are expressible in terms of zeta values, the Dirichlet beta values, linear $R$-sums and linear $\S$-sums. Note further that by Eq. (\ref{Stpq.DMTV}), all the linear $\S$-sums can be rewritten as double $T$-values. Thus, the assertion of the theorem holds.
\end{proof}

\begin{remark}
According to Theorem \ref{Th.Rpq.MTV}, if $(p,\si_1)\neq(1,1)$, the linear $R$-sums $R_{p,q}^{\si_1,\si_2}$ can be rewritten as combinations of zeta values, the Dirichlet beta values and double $T$-values. In the case of $(p,\si_1)=(1,1)$, we have $R_{1,q}$ and $R_{1,\bar{q}}$. By Eq. (\ref{R1q}), the sums $R_{1,q}$ are expressible in terms of zeta values, and by Eq. (\ref{Rpq}), the sums $R_{1,\bar{q}}$, where $q$ is odd, are also reducible to zeta values and beta values. We do not know whether the linear $R$-sums $R_{1,\bar{q}}$, where $q$ is even, can be converted into combinations of zeta values, beta values and double $T$-values, though by Theorem \ref{Th.Rpq.CMZV}, they can be expressed by CMZVs. As a result, we do not know, under the conditions of Theorem \ref{Th.PT.MTV}, whether the triple $T$-values $T(m,p,q;\si_1,\si_2,\si_3)$ can reduce to combinations of zeta values, beta values and double $T$-values directly.\hfill\qedsymbol
\end{remark}

Using Eqs. (\ref{Stpq.DMTV}) and (\ref{Stmpq.TMTV}) as well as the results of the Euler $R$-sums and $\S$-sums, we can find the evaluations and reduction formulas of some Kaneko-Tsumura double and triple $T$-values.

\begin{example}
Here are some selected Kaneko-Tsumura triple $T$-values:
\begin{align*}
T(2,\bar{1},1)&=2\pi\imm(\Li_3(\tf{1}{2}+\tf{\ui}{2}))+\pi\ln(2)G
    -\tf{1}{16}\pi^2\ln(2)^2-\tf{3}{64}\pi^4\,,\\
T(\bar{2},\bar{1},\bar{1})&=-4G^2+2\pi\imm(\Li_3(\tf{1}{2}+\tf{\ui}{2}))
    +\pi\ln(2)G-\tf{1}{16}\pi^2\ln(2)^2-\tf{1}{64}\pi^4\,,\\
T(3,2,1)&=\tf{49}{16}\ze(3)^2+2\pi^2\Lii_4(\tf{1}{2})+\tf{7}{4}\pi^2\ln(2)\ze(3)
    +\tf{1}{12}\pi^2\ln(2)^4-\tf{1}{12}\pi^4\ln(2)^2-\tf{151}{5760}\pi^6\,,\\
T(\bar{3},2,\bar{1})&=-12\be(4)G-\pi^2\Lii_4(\tf{1}{2})-\tf{7}{8}\pi^2\ln(2)\ze(3)
    +\tf{1}{2}\pi^2G^2-\tf{1}{24}\pi^2\ln(2)^4\\
&\quad+\tf{1}{4}\pi^3\imm(\Li_3(\tf{1}{2}+\tf{\ui}{2}))
    +\tf{1}{8}\pi^3\ln(2)G+\tf{13}{384}\pi^4\ln(2)^2+\tf{257}{23040}\pi^6\,,
\end{align*}
and
\begin{align*}
T(2,2,\bar{1})&=-7\ze(3)G+2\pi\Lii_4(\tf{1}{2})+\tf{7}{4}\pi\ln(2)\ze(3)
    +\pi G^2+\tf{1}{12}\pi\ln(2)^4+\pi^2\imm(\Li_3(\tf{1}{2}+\tf{\ui}{2}))\\
&\quad+\tf{1}{2}\pi^2G\ln(2)
    -\tf{11}{96}\pi^3\ln(2)^2-\tf{49}{1440}\pi^5\,,\\
T(\bar{2},2,1)&=\tf{7}{2}\ze(3)G+2\pi\Lii_4(\tf{1}{2})+\tf{7}{4}\pi\ln(2)\ze(3)
    -\pi G^2+\tf{1}{12}\pi\ln(2)^4\\
&\quad-\tf{1}{12}\pi^3\ln(2)^2-\tf{151}{5760}\pi^5\,,\\
T(2,\bar{2},\bar{1})&=7\ze(3)G-3\pi^2\imm(\Li_3(\tf{1}{2}+\tf{\ui}{2}))
    -\tf{3}{2}\pi^2\ln(2)G+\tf{3}{32}\pi^3\ln(2)^2+\tf{7}{128}\pi^5\,,\\
T(\bar{2},\bar{2},1)&=-\tf{7}{2}\ze(3)G-\pi^2\imm(\Lii_3(\tf{1}{2}+\tf{\ui}{2}))
    -\tf{1}{2}\pi^2\ln(2)G+\tf{1}{32}\pi^3\ln(2)^2+\tf{5}{128}\pi^5\,.
\end{align*}
Next, we give two representative reduction formulas:
\begin{align*}
T(4,\bar{1},1)&=\tf{1}{2}\pi T(3,\bar{2})+\tf{1}{2}\pi T(2,\bar{3})
    +\tf{1}{16}\pi R_{1,\bar{4}}+2\pi\ln(2)\be(4)
    +\tf{1}{2}\pi^3\imm(\Li_3(\tf{1}{2}+\tf{\ui}{2}))\\
&\quad+\tf{1}{4}\pi^3\ln(2)G
    -\tf{1}{64}\pi^4\ln(2)^2-\tf{11}{768}\pi^6\,,\\
T(5,\bar{1},\bar{1})&=-\tf{31}{8}\ze(5)G
    +\tf{1}{4}\pi^2T(\bar{4},\bar{1})+\tf{1}{4}\pi^2T(\bar{3},\bar{2})
    +\tf{1}{4}\pi^2T(\bar{2},\bar{3})+\tf{1}{4}\pi^2T(\bar{1},\bar{4})\\
&\quad-\tf{1}{12}\pi^4\imm(\Li_3(\tf{1}{2}+\tf{\ui}{2}))
    -\tf{1}{24}\pi^4\ln(2)G+\tf{1}{384}\pi^5\ln(2)^2+\tf{3}{512}\pi^7\,.
\end{align*}
More evaluations and reduction formulas of the Kaneko-Tsumura double and triple $T$-values can be obtained by the results of the present paper.\hfill\qedsymbol
\end{example}


\section{Further remarks}\label{Sec.Remark}

It is clear that some main results in our previous paper \cite{XuWang19.TVE} are immediate corollaries of this paper, and the four types of the Euler sums
\[
S_{p_1,p_2,\ldots,p_k,q}^{\si_1,\si_2,\ldots,\si_k,\si}\,,\quad T_{p_1,p_2,\ldots,p_k,q}^{\si_1,\si_2,\ldots,\si_k,\si}\,,\quad \S_{p_1,p_2,\ldots,p_k,q}^{\si_1,\si_2,\ldots,\si_k,\si}\,,\quad R_{p_1,p_2,\ldots,p_k,q}^{\si_1,\si_2,\ldots,\si_k,\si}
\]
can be expressed in terms of linear combinations of CMZVs of level four.

In addition, it is possible to establish some other relations involving alternating Euler $T$-sums and $\S$-sums by using the techniques described here. Let $A,A^{(1)},\ldots,A^{(k)}\in\{A_1,A_2\}$, where $A_1=\{1^k\}$ and $A_2=\{(-1)^k\}$. Then we have the following two parity theorems.

\begin{theorem}\label{Th.redu.Tsum}
The Euler $T$-sums $T_{p_1,p_2,\ldots,p_k,q}^{\si_1,\si_2,\ldots,\si_k,\si}$ of degree $k$ are reducible to (alternating) Euler sums and Euler $T$-sums of lower degrees, whenever the sum $p_1+\cdots+p_k+q+k$ and the quantity $\frac{1-\si_1\cdots\si_k\si}{2}$ are of the same parity.
\end{theorem}

\begin{proof}
Applying the kernel function
\[
\xi(s)=\pi\cot(\pi s;A)
    \prod_{i=1}^k\frac{\vPsi^{(p_i-1)}(\tf{1}{2}-s;A^{(i)})}{(p_i-1)!}
\]
to the base function $r_1(s)=(s-1/2)^{-q}$, and considering the residues, we have
\begin{align*}
&(-1)^{p_1+p_2+\cdots+p_k+q}\sum_{n=1}^\infty
    \frac{a_{n-1}\bar{N}^{(A^{(1)})}_n(p_1)\bar{N}^{(A^{(2)})}_n(p_2)
        \cdots\bar{N}^{(A^{(k)})}_n(p_k)}{(n-1/2)^q}\\
&\quad+\sum_{n=1}^\infty
    \frac{a_nN^{(A^{(1)})}_n(p_1)N^{(A^{(2)})}_n(p_2)
        \cdots N^{(A^{(k)})}_n(p_k)}{(n-1/2)^q}\\
&\quad=\text{a combination of general sums}\,,
\end{align*}
where $a_n$ is the $n$th term in the sequence $A$. Now, set $A=\{(\si_1\si_2\cdots\si_k\si)^n\}_{n\in\mathbb{Z}}$ and $A^{(i)}=\{\si_i^n\}_{n\in\mathbb{Z}}$, for $i=1,2,\ldots,k$, and perform transformations. Then the result will be of the form
\[
(1+\si_1\si_2\cdots\si_k\si(-1)^{p_1+p_2+\cdots+p_k+q+k})
    T_{p_1,p_2,\ldots,p_k,q}^{\si_1,\si_2,\ldots,\si_k,\si}=U\,,
\]
where $U$ is a combination of sums of lower degrees. Thus, we obtain the desired description.
\end{proof}

\begin{theorem}\label{Th.redu.Stsum}
The Euler $\S$-sums $\S_{p_1,p_2,\ldots,p_k,q}^{\si_1,\si_2,\ldots,\si_k,\si}$ of degree $k$ are reducible to (alternating) Euler $\S$-sums and Euler $R$-sums of lower degrees, whenever the sum $p_1+\cdots+p_k+q+k$ and the quantity $\frac{1-\si_1\cdots\si_k}{2}$ are of the same parity.
\end{theorem}

\begin{proof}
In this case, applying the kernel function $\xi(s)$ to the base function $r_2(s)=s^{-q}$, and following along the same lines, we have
\begin{align*}
&(-1)^{p_1+p_2+\cdots+p_k+q}\sum_{n=1}^\infty
    \frac{a_n\bar{N}^{(A^{(1)})}_{n+1}(p_1)\bar{N}^{(A^{(2)})}_{n+1}(p_2)
        \cdots\bar{N}^{(A^{(k)})}_{n+1}(p_k)}{n^q}\\
&\quad+\sum_{n=1}^\infty\frac{a_nN^{(A^{(1)})}_n(p_1)N^{(A^{(2)})}_n(p_2)
    \cdots N^{(A^{(k)})}_n(p_k)}{n^q}\\
&\quad=\text{a combination of general sums}\,.
\end{align*}
Next, set $A=\{(\si_1\si_2\cdots\si_k\si)^n\}_{n\in\mathbb{Z}}$ and $A^{(i)}=\{\si_i^n\}_{n\in \mathbb{Z}}$, for $i=1,2,\ldots,k$. Then we have
\[
(1+\si_1\si_2\cdots\si_k(-1)^{p_1+p_2+\cdots+p_k+q+k})
    \S_{p_1,p_2,\ldots,p_k,q}^{\si_1,\si_2,\ldots,\si_k,\si}=V\,,
\]
where $V$ is a combination of sums of lower degrees, and the assertion of the theorem follows.
\end{proof}


\section*{Acknowledgments}
\addcontentsline{toc}{section}{Acknowledgments}

The first author Ce Xu is supported by the National Natural Science Foundation of China (under Grant 12101008) and the Scientific Research Foundation for Scholars of Anhui Normal University. The corresponding author Weiping Wang is supported by the Zhejiang Provincial Natural Science Foundation of China (under Grant LY22A010018) and the National Natural Science Foundation of China (under Grant 11671360).



\section*{Appendix}\label{Sec.Appendix}
\addcontentsline{toc}{section}{Appendix}
\renewcommand{\thesubsection}{A.\arabic{subsection}}


\subsection{Some linear $T$-sums of weights $8$ and $9$}

For the linear sums $T_{p,q}$ of weight $9$, we have
\begin{align*}
&T_{1,8}=-\tf{511}{2}\ze(9)-\tf{1}{2}\pi^2\ze(7)-\tf{1}{6}\pi^4\ze(5)
    -\tf{1}{15}\pi^6\ze(3)+\tf{17}{315}\pi^8\ln(2)\,,\\
&T_{2,7}=-\tf{511}{2}\ze(9)+\tf{7}{2}\pi^2\ze(7)+\tf{2}{3}\pi^4\ze(5)+\tf{2}{15}\pi^6\ze(3)\,,\\
&T_{3,6}=-\tf{511}{2}\ze(9)-\tf{21}{2}\pi^2\ze(7)-\pi^4\ze(5)+\tf{2}{5}\pi^6\ze(3)\,,\\
&T_{4,5}=-\tf{511}{2}\ze(9)+\tf{35}{2}\pi^2\ze(7)+\tf{5}{6}\pi^4\ze(5)\,,\\
&T_{5,4}=-\tf{511}{2}\ze(9)-\tf{35}{2}\pi^2\ze(7)+\tf{13}{3}\pi^4\ze(5)\,,\\
&T_{6,3}=-\tf{511}{2}\ze(9)+\tf{21}{2}\pi^2\ze(7)+\pi^4\ze(5)+\tf{1}{15}\pi^6\ze(3)\,,\\
&T_{7,2}=-\tf{511}{2}\ze(9)+60\pi^2\ze(7)-\tf{2}{3}\pi^4\ze(5)-\tf{2}{15}\pi^6\ze(3)\,.
\end{align*}
For the linear sums $T_{\bar{p},\bar{q}}$ of weight $9$, we have
\begin{align*}
&T_{\bar{1},\bar{8}}=-\tf{511}{2}\ze(9)+\tf{63}{128}\pi^2\ze(7)
    +\tf{5}{32}\pi^4\ze(5)+\tf{1}{20}\pi^6\ze(3)+\tf{17}{630}\pi^8\ln(2)\,,\\
&T_{\bar{2},\bar{7}}=-\tf{511}{2}\ze(9)-\tf{441}{128}\pi^2\ze(7)
    -\tf{5}{8}\pi^4\ze(5)-\tf{1}{10}\pi^6\ze(3)+\tf{61}{360}\pi^7G\,,\\
&T_{\bar{3},\bar{6}}=-\tf{511}{2}\ze(9)+\tf{1323}{128}\pi^2\ze(7)
    +\tf{15}{16}\pi^4\ze(5)+\tf{1}{20}\pi^6\ze(3)\,,\\
&T_{\bar{4},\bar{5}}=-\tf{511}{2}\ze(9)-\tf{2205}{128}\pi^2\ze(7)
    -\tf{25}{32}\pi^4\ze(5)+\tf{5}{3}\pi^5\be(4)\,,\\
&T_{\bar{5},\bar{4}}=-\tf{511}{2}\ze(9)+\tf{2205}{128}\pi^2\ze(7)+\tf{25}{32}\pi^4\ze(5)\,,\\
&T_{\bar{6},\bar{3}}=-\tf{511}{2}\ze(9)-\tf{1323}{128}\pi^2\ze(7)+16\pi^3\be(6)
    -\tf{15}{16}\pi^4\ze(5)-\tf{1}{20}\pi^6\ze(3)\,,\\
&T_{\bar{7},\bar{2}}=-\tf{511}{2}\ze(9)+\tf{441}{128}\pi^2\ze(7)
    +\tf{5}{8}\pi^4\ze(5)+\tf{1}{10}\pi^6\ze(3)\,,\\
&T_{\bar{8},\bar{1}}=-\tf{511}{2}\ze(9)+128\pi\be(8)-\tf{63}{128}\pi^2\ze(7)
    -\tf{5}{32}\pi^4\ze(5)-\tf{1}{20}\pi^6\ze(3)-\tf{17}{630}\pi^8\ln(2)\,.
\end{align*}
For the linear sums $T_{p,\bar{q}}$ of weight $8$, we have
\begin{align*}
&T_{1,\bar{7}}=-128\be(8)-\tf{127}{128}\pi\ze(7)-\tf{1}{4}\pi^3\ze(5)
    -\tf{5}{48}\pi^5\ze(3)+\tf{61}{720}\pi^7\ln(2)\,,\\
&T_{2,\bar{6}}=-128\be(8)+\tf{381}{64}\pi\ze(7)+\pi^3\ze(5)+\tf{5}{24}\pi^5\ze(3)\,,\\
&T_{3,\bar{5}}=-128\be(8)-\tf{1905}{128}\pi\ze(7)-\tf{111}{64}\pi^3\ze(5)+\tf{5}{8}\pi^5\ze(3)\,,\\
&T_{4,\bar{4}}=-128\be(8)+\tf{635}{32}\pi\ze(7)+\tf{31}{16}\pi^3\ze(5)\,,\\
&T_{5,\bar{3}}=-128\be(8)-\tf{1905}{128}\pi\ze(7)
    +\tf{195}{32}\pi^3\ze(5)-\tf{5}{64}\pi^5\ze(3)\,,\\
&T_{6,\bar{2}}=-128\be(8)+\tf{381}{64}\pi\ze(7)
    +\tf{15}{16}\pi^3\ze(5)+\tf{5}{32}\pi^5\ze(3)\,,\\
&T_{7,\bar{1}}=-128\be(8)+\tf{8001}{128}\pi\ze(7)-\tf{15}{64}\pi^3\ze(5)
    -\tf{5}{64}\pi^5\ze(3)-\tf{61}{1440}\pi^7\ln(2)\,.
\end{align*}
For the linear sums $T_{\bar{p},q}$ of weight $8$, we have
\begin{align*}
&T_{\bar{1},7}=-128\be(8)+\tf{127}{128}\pi\ze(7)+\tf{15}{64}\pi^3\ze(5)
    +\tf{5}{64}\pi^5\ze(3)+\tf{61}{1440}\pi^7\ln(2)\,,\\
&T_{\bar{2},6}=-128\be(8)-\tf{381}{64}\pi\ze(7)-\tf{15}{16}\pi^3\ze(5)
    -\tf{5}{32}\pi^5\ze(3)+\tf{4}{15}\pi^6G\,,\\
&T_{\bar{3},5}=-128\be(8)+\tf{1905}{128}\pi\ze(7)+\tf{53}{32}\pi^3\ze(5)+\tf{5}{64}\pi^5\ze(3)\,,\\
&T_{\bar{4},4}=-128\be(8)-\tf{635}{32}\pi\ze(7)-\tf{31}{16}\pi^3\ze(5)+\tf{8}{3}\pi^4\be(4)\,,\\
&T_{\bar{5},3}=-128\be(8)+\tf{1905}{128}\pi\ze(7)
    +\tf{111}{64}\pi^3\ze(5)+\tf{5}{48}\pi^5\ze(3)\,,\\
&T_{\bar{6},2}=-128\be(8)-\tf{381}{64}\pi\ze(7)+32\pi^2\be(6)-\pi^3\ze(5)-\tf{5}{24}\pi^5\ze(3)\,.
\end{align*}


\subsection{Some linear $\S$-sums of weights $8$ and $9$}

For the linear sums $\S_{p,q}$ of weight $9$, we have
\begin{align*}
&\S_{1,8}=\tf{511}{2}\ze(9)-\tf{127}{6}\pi^2\ze(7)-\tf{31}{90}\pi^4\ze(5)
    -\tf{1}{135}\pi^6\ze(3)\,,\\
&\S_{2,7}=-2044\ze(9)+\tf{381}{2}\pi^2\ze(7)+\tf{62}{45}\pi^4\ze(5)+\tf{2}{135}\pi^6\ze(3)\,,\\
&\S_{3,6}=7154\ze(9)-\tf{1397}{2}\pi^2\ze(7)-\tf{31}{15}\pi^4\ze(5)\,,\\
&\S_{4,5}=-14308\ze(9)+\tf{8255}{6}\pi^2\ze(7)+\tf{589}{90}\pi^4\ze(5)\,,\\
&\S_{5,4}=17885\ze(9)-\tf{3175}{2}\pi^2\ze(7)-\tf{62}{3}\pi^4\ze(5)\,,\\
&\S_{6,3}=-14308\ze(9)+\tf{2159}{2}\pi^2\ze(7)+31\pi^4\ze(5)+\tf{7}{15}\pi^6\ze(3)\,,\\
&\S_{7,2}=7154\ze(9)-381\pi^2\ze(7)-\tf{62}{3}\pi^4\ze(5)-\tf{14}{15}\pi^6\ze(3)\,.
\end{align*}
For the linear sums $\S_{p,\bar{q}}$ of weight $9$, we have
\begin{align*}
&\S_{1,\bar{8}}=-\tf{511}{2}\ze(9)+128\pi\be(8)-\tf{127}{12}\pi^2\ze(7)
    -\tf{217}{720}\pi^4\ze(5)-\tf{31}{4320}\pi^6\ze(3)\,,\\
&\S_{2,\bar{7}}=2044\ze(9)-896\pi\be(8)+\tf{127}{2}\pi^2\ze(7)
    +\tf{217}{180}\pi^4\ze(5)+\tf{31}{2160}\pi^6\ze(3)\,,\\
&\S_{3,\bar{6}}=-7154\ze(9)+2688\pi\be(8)-\tf{635}{4}\pi^2\ze(7)+16\pi^3\be(6)
    -\tf{217}{120}\pi^4\ze(5)\,,\\
&\S_{4,\bar{5}}=14308\ze(9)-4480\pi\be(8)+\tf{635}{3}\pi^2\ze(7)-80\pi^3\be(6)
    +\tf{217}{180}\pi^4\ze(5)\,,\\
&\S_{5,\bar{4}}=-17885\ze(9)+4480\pi\be(8)-\tf{635}{4}\pi^2\ze(7)
    +160\pi^3\be(6)+\tf{5}{3}\pi^5\be(4)\,,\\
&\S_{6,\bar{3}}=14308\ze(9)-2688\pi\be(8)+\tf{127}{2}\pi^2\ze(7)
    -160\pi^3\be(6)-5\pi^5\be(4)\,,\\
&\S_{7,\bar{2}}=-7154\ze(9)+896\pi\be(8)+80\pi^3\be(6)
    +5\pi^5\be(4)+\tf{61}{360}\pi^7G\,,\\
&\S_{8,\bar{1}}=2044\ze(9)-128\pi\be(8)-16\pi^3\be(6)-\tf{5}{3}\pi^5\be(4)
    -\tf{61}{360}\pi^7G\,.
\end{align*}
For the linear sums $\S_{\bar{p},q}$ of weight $8$, we have
\begin{align*}
&\S_{\bar{1},7}=-128\be(8)+\tf{127}{2}\pi\ze(7)-\tf{16}{3}\pi^2\be(6)
    -\tf{7}{45}\pi^4\be(4)-\tf{31}{7560}\pi^6G\,,\\
&\S_{\bar{2},6}=896\be(8)-381\pi\ze(7)+\tf{80}{3}\pi^2\be(6)
    +\tf{7}{15}\pi^4\be(4)+\tf{1}{120}\pi^6G\,,\\
&\S_{\bar{3},5}=-2688\be(8)+\tf{1905}{2}\pi\ze(7)-\tf{160}{3}\pi^2\be(6)
    +\tf{31}{4}\pi^3\ze(5)-\tf{7}{15}\pi^4\be(4)\,,\\
&\S_{\bar{4},4}=4480\be(8)-1270\pi\ze(7)+\tf{160}{3}\pi^2\be(6)-31\pi^3\ze(5)
    +\tf{1}{3}\pi^4\be(4)\,,\\
&\S_{\bar{5},3}=-4480\be(8)+\tf{1905}{2}\pi\ze(7)-\tf{80}{3}\pi^2\be(6)
    +\tf{93}{2}\pi^3\ze(5)+\tf{35}{48}\pi^5\ze(3)\,,\\
&\S_{\bar{6},2}=2688\be(8)-381\pi\ze(7)+16\pi^2\be(6)-31\pi^3\ze(5)-\tf{35}{24}\pi^5\ze(3)\,.
\end{align*}
For the linear sums $\S_{\bar{p},\bar{q}}$ of weight $8$, we have
\begin{align*}
&\S_{\bar{1},\bar{7}}=128\be(8)-\tf{32}{3}\pi^2\be(6)
    -\tf{8}{45}\pi^4\be(4)-\tf{4}{945}\pi^6G\,,\\
&\S_{\bar{2},\bar{6}}=-896\be(8)+\tf{256}{3}\pi^2\be(6)+\tf{8}{15}\pi^4\be(4)
    +\tf{1}{120}\pi^6G\,,\\
&\S_{\bar{3},\bar{5}}=2688\be(8)-\tf{800}{3}\pi^2\be(6)-\tf{8}{15}\pi^4\be(4)\,,\\
&\S_{\bar{4},\bar{4}}=-4480\be(8)+\tf{1280}{3}\pi^2\be(6)+3\pi^4\be(4)\,,\\
&\S_{\bar{5},\bar{3}}=4480\be(8)-\tf{1120}{3}\pi^2\be(6)-8\pi^4\be(4)\,,\\
&\S_{\bar{6},\bar{2}}=-2688\be(8)+176\pi^2\be(6)+8\pi^4\be(4)+\tf{4}{15}\pi^6G\,,\\
&\S_{\bar{7},\bar{1}}=896\be(8)-32\pi^2\be(6)-\tf{8}{3}\pi^4\be(4)-\tf{4}{15}\pi^6G\,.
\end{align*}
\end{document}